\documentclass[oneside,a4paper,11pt]{amsart}
\usepackage{amsfonts}
\usepackage{amssymb}
\usepackage{amsmath}
\usepackage[height=9in,a4paper,hmargin={3.5cm,1in}]{geometry}
\usepackage{epsfig}
\usepackage{graphicx}
\usepackage{enumitem}
\usepackage{verbatim}
\usepackage[T1]{fontenc}
\usepackage[utf8]{inputenc}

\usepackage{graphicx}
\usepackage{amsthm}
\usepackage[english]{babel}
\usepackage{amsfonts}
\usepackage{listings}
\usepackage{caption}
\usepackage{csquotes}
\usepackage{amsmath}
\usepackage{amssymb}
\usepackage{mathtools}
\usepackage{tabularx}
\usepackage{enumitem}
\usepackage{bm}
\usepackage{tikz}
\tikzset{
	mynode/.style={fill,circle,inner sep=1.5pt,outer sep=0pt}
}
\usepackage{pgfplots} 
\pgfplotsset{/pgf/number format/use comma,compat=newest}
\theoremstyle{plain}
\newtheorem{thm}{Theorem}[section]
\newtheorem{cor}[thm]{Corollary}
\newtheorem{lem}[thm]{Lemma}
\newtheorem{prop}[thm]{Proposition}
\theoremstyle{definition}
\newtheorem{defn}[thm]{Definition}
\newtheorem{rem}[thm]{Remark}

\newcommand{\R}{\mathbb{R}}

\newcommand{\N}{\mathbb{N}}

\newcommand{\Int}{\textnormal{Int}}
\newcommand{\dist}{\textnormal{dist}}

\newcommand{\diam}{\textnormal{diam}}

\newcommand{\sbvp}{\textnormal{SBV}^p}
\newcommand{\sbv}{\textnormal{SBV}}
\newcommand{\gsbv}{\textnormal{GSBV}}
\newcommand{\gsbvp}{\textnormal{GSBV}^p}
\newcommand {\Aff}{\textnormal{Aff}}

\usepackage{tikz}
\usetikzlibrary{decorations.pathreplacing}

\definecolor{blue_links}{RGB}{13,0,180} 

\usepackage{hyperref}
\hypersetup{
	colorlinks=true, 
	linktoc=all,     
	linkcolor=blue_links,  
	citecolor=blue_links,
	urlcolor=blue_links,
}

\usepackage{mathtools}

\usepackage{color}

\makeatletter
\@namedef{subjclassname@2020}{%
	\textup{2020} Mathematics Subject Classification}
\makeatother

\title[Brittle membranes in finite elasticity]{Brittle membranes in finite elasticity}

\author[S. Almi]{Stefano Almi}
\address[Stefano Almi]{Dipartimento di Matematica e Applicazioni ``R.~Caccioppoli'',
	Universit\`a di Napoli Federico II, via Cintia, 80126 Napoli, Italy.}
\email{stefano.almi@unina.it}

\author[D. Reggiani]{Dario Reggiani}
\address[Dario Reggiani]{Scuola Superiore Meridionale, Universit\'a di Napoli Federico II
	Largo San Marcellino 10, 80138 Napoli, Italy.}
\email{dario.reggiani@unina.it}

\author[F. Solombrino]{Francesco Solombrino}
\address[Francesco Solombrino]{Dipartimento di Matematica e Applicazioni ``R.~Caccioppoli'',
	Universit\`a di Napoli Federico II, via Cintia, 80126 Napoli, Italy.}
\email{francesco.solombrino@unina.it}

\begin{document}	
	
	\maketitle
	
	\begin{abstract}
		\noindent This work is devoted to the variational derivation of a reduced model for brittle membranes in finite elasticity. The main mathematical tools we develop for our analysis are: (i) a new density result in $GSBV^{p}$ of functions satisfying a maximal-rank constraint on the subgradients, which can be approximated by $C^{1}$-local immersions on regular subdomains of the cracked set, and (ii) the construction of recovery sequences by means of suitable $W^{1,\infty}$ diffeomorphisms mapping the regular subdomains onto the fractured configuration.
	\end{abstract}
	
	\section{Introduction}

	This paper focuses on the derivation of a reduced model for brittle elastic membranes under non-interpenetration constraints. Our approach is based on~$\Gamma$-convergence~\cite{DM93gamma}, which has been successfully used in the last decades to tackle a number of dimension reduction problems in elasticity for rods and ribbons~\cite{Cicalese2017-2, Cicalese2017, Davoli2013-2, Friedrich2022, Mora2003, Mora2004}, membranes~\cite{Hafsa2005TheNM, hafsa2006nonlinear, Anzellotti1994, Cesana, LeDret1995, LeDret1996, Trabelsimodeling}, and plates~\cite{Conti2009, FriedrichKruzik2020, Friesecke2002ATO, Friesecke2006, Hornung2014, Lewicka2011}, as well as in elastoplasticity~\cite{Davoli2014-2, Davoli2014, Davoli2013, Maggiani2018}.
	
	In the framework of variational fracture mechanics~\cite{Bourdin2008TheVA}, few dimension reduction results have been obtained in linearized~\cite{Almi2021, Almi2021-2, Babadjian2016, Freddi2010, Gladbach2022, Baldelli2014} and nonlinear elasticity~\cite{Babadjian2006QuasistaticEO, schmidt2016griffitheulerbernoulli}. In the mentioned papers an explicit non-interpenetration constraint in the form of positive sign of the Jacobian of the deformation was not considered in the membranes setting. In order to fill this gap, we propose here to study the limit as $\rho \to 0$ of the $3$-dimensional functional
	\begin{equation}\label{main functional}
		\mathcal{F}_\rho(v) :=
		\begin{cases*}
			\displaystyle \int_{\Sigma_\rho} W(\nabla v) \, dy +\mathcal{H}^2(J_v) & if $v \in \gsbvp(\Sigma_\rho,\R^3)$, \\
			+ \infty & otherwise in~$L^{0}(\Sigma_{\rho}, \R^{3})$.
		\end{cases*}
	\end{equation}
	defined on the thin reference configuration~$\Sigma_{\rho}:= \Sigma \times (-\frac{\rho}{2}, \frac{\rho}{2})$ of thickness $\rho>0$ and middle surface~$\Sigma$, an open bounded subset of~$\R^{2}$ with Lipschitz boundary~$\partial\Sigma$. In~\eqref{main functional}, $L^{0}(\Sigma_{\rho}, \R^{3})$ is the space of measurable functions on~$\Sigma_{\rho}$ with values in~$\R^{3}$ and $\gsbvp(\Sigma_{\rho}, \R^{3})$ for $p \in (1, +\infty)$ denotes the space of {\em generalized special functions of bounded variations} \cite[Section4.5]{Ambrosio2000FunctionsOB} whose approximate gradient~$\nabla{v}$ belongs to~$L^{p}(\Sigma_{\rho}, \mathbb{M}^{3 \times 3})$ and whose jump set~$J_{v}$ has finite $2$-dimensional Hausdorff measure~$\mathcal{H}^{2}(J_{v})$. The stored elastic energy density~$W \colon \mathbb{M}^{3 \times 3} \to [0, +\infty]$ is assumed to satisfy the following conditions: 
	\begin{itemize}
		\item[($H_1$)] $W\in C(\mathbb{M}^{3 \times 3}; [0, +\infty])$;
		\item[($H_2$)] $W$ is frame indifferent, that is
		$$
		W(F) = W(RF) \ \ \  \mbox{for every $F \in \mathbb{M}^{3 \times 3}$ and $R \in SO(3)$;}
		$$ 
		\item[($H_3$)] for every $F \in \mathbb{M}^{3 \times 3}$
		\begin{equation*}
			W(F) < +\infty \ \ \mbox{if }  \det F > 0, \ \mbox{and} \ W(F) = +\infty \ \ \mbox{if }  \det F \leq 0;
		\end{equation*}
		\item[($H_4$)] there exists $C_1>0$ such that for every $F \in \mathbb{M}^{3 \times 3}$
		\begin{align*}
			W(F) \geq C_1|F|^p - \frac{1}{C_1}
		\end{align*}
		\item[($H_5$)] for every $\delta>0$ there exists $c_\delta>0$ such that for every $F \in \mathbb{M}^{3 \times 3}$ with $\det F \geq \delta$
		\begin{align*}
			W(F) \leq c_\delta(1 + |F|^p).
		\end{align*}
	\end{itemize}
	As usual in finite elasticity, ($H_{3}$) forces the non-interpenetration constraint $\det \nabla v>0$ a.e.~in~$\Sigma_{\rho}$ for every $v \in GSBV^{p}(\Sigma_{\rho}, \R^{3})$ with~$\mathcal{F}_{\rho}(v)<+\infty$. Moreover, we notice that~($H_{3}$) is incompatible with a $p$-growth condition. Nevertheless, condition~($H_{5}$) ensures that $W$ satisfies a weak $p$-growth condition which degenerates as $\det F$ approaches $0$ (see also~\cite{Hafsa2005TheNM, hafsa2006nonlinear}). 
	
	In order to reduce ourselves to a fixed domain, we perform the change of variables $(y_1,y_2,y_3)=(x_1,x_2,\rho x_3)$ in \eqref{main functional} and define $u \in L^0(\Sigma_1,\R^3)$ as $u(x)=v(x_1,x_2,\rho x_3)$ for $v \in L^0(\Sigma_\rho,\R^3)$. This amounts to rewrite the functional $\mathcal{F}_\rho$ as
	\begin{equation}
		\label{e:Grho}
		\rho^{-1}\mathcal{G}_\rho(u) := 
		\begin{cases*}
			\displaystyle
			\int_{\Sigma_1} W \left(\nabla_\alpha u \left|  \frac{1}{\rho} \partial_3 u \right. \right) dx +\int_{J_u} \psi_{\rho}(\nu_{u}) \, d \mathcal{H}^2 & if $u \in \gsbvp(\Sigma_1,\R^3)$, \\[1mm]
			+ \infty & otherwise in $L^{0}(\Sigma_{1}, \R^{3})$,
		\end{cases*}
	\end{equation} 
	where $\nabla_{\alpha} u := (\partial_{1} u | \partial_{2} u) \in \mathbb{M}^{3 \times 2}$, $\nu_{u} \in \mathbb{S}^{2}$ denotes the approximate unit normal to~$J_{u}$, and
	\begin{displaymath}
		\psi_{\rho} (\nu) := \left| \left( \nu_{1} , \nu_{2}, \frac{1}{\rho} \nu_{3} \right)\right| \qquad \text{for every $\nu \in \R^{3}$ and every $\rho>0$.}
	\end{displaymath}
	
	The main difficulty in the identification of the $\Gamma$-limit of the sequence~$\rho^{-1} \mathcal{G}_{\rho}$ is indeed the lack of an uniform $p$-growth of the bulk energy density~$W$, which prevents us from directly applying classical representation theorems for free discontinuity functionals~\cite{Bouchitte2002, Braides1996}. 
	
	When fracture is not present, the $\Gamma$-limit of the elastic energy
	\begin{equation}
		\label{intro:elastic-part}
		\mathcal{I}_{\rho} (u) := 
		\int_{\Sigma_1} W \left(\nabla_\alpha u \left|  \frac{1}{\rho} \partial_3 u \right. \right)  dx  \qquad \text{for $u \in W^{1, p}(\Sigma_1,\R^3)$}
	\end{equation}
	has been identified in~\cite{Hafsa2005TheNM, hafsa2006nonlinear} and is given by
	\begin{equation*}
		\mathcal{I}_0(u) := \int_{\Sigma_1} \mathcal{Q} W_0(\nabla_\alpha u) \, dx  \qquad \text{for $u \in W^{1, p}(\Sigma_{1}, \R^{3})$ with $\partial_{3} u=0$,}
	\end{equation*}
	where the reduced elastic energy density $W_{0} \colon \mathbb{M}^{3 \times 2} \to [0,+\infty]$ is defined as
	\begin{equation}\label{reduced density}
		W_0(E):= \inf_{\xi \in \R^3}W(E|\xi) \qquad \mbox{for} \ E \in \mathbb{M}^{3 \times 2}
	\end{equation}
	and $ \mathcal{Q} W_0$ denotes the quasi-convex envelope of~$W_{0}$. The proof of such $\Gamma$-convergence is based on the density in Sobolev spaces of local immersions, i.e., of $C^{1}$-functions satisfying a maximum-rank condition on the Jacobian. Such density result builds upon a classical theorem by Gromov and Eliashberg (see~\cite{gromov1986partial, gromov1971construction}), which holds on regular domains. A first step towards the $\Gamma$-limsup inequality in~\cite{Hafsa2005TheNM, hafsa2006nonlinear} is indeed the characterization of the $\Gamma$-limit of~$\mathcal{I}_{\rho}$ as the relaxation of the functional
	\begin{displaymath}
		\int_{\Sigma} W_{0} (\nabla u) \, dx \qquad \text{for $u \in W^{1, p}(\Sigma, \R^{3})$ local immersion.}
	\end{displaymath}
	Then, a local construction of an optimizer (see~\cite[Lemma 5.4]{hafsa2006nonlinear} and~\cite{benbe1996}) allows one to estimate the $\Gamma$-limsup of~$\mathcal{I}_{\rho}$ with the functional
	\begin{displaymath}
		\int_{\Sigma} \mathcal{R}W_{0} (\nabla{u}) \, dx \qquad u \in W^{1, p}(\Sigma, \R^{3})\,,
	\end{displaymath}
	where $\mathcal{R} W_{0}$ indicates the rank-one convex envelope of~$W_{0}$. In particular, it turns out that~$\mathcal{R}W_{0}$ has $p$-growth (see~\cite[Lemma 5.2]{hafsa2006nonlinear} and Lemma~\ref{lemma rank one} below). Thus, the aforementioned density of local immersions and standard relaxation results in Sobolev spaces entail the desired $\Gamma$-convergence.
	
	Coming to brittle fracture, the presence of the jump set makes the identification of the $\Gamma$-limit of~$\rho^{-1}\mathcal{G}_{\rho}$ more involved. Our aim is to show that the sequence~$\rho^{-1}\mathcal{G}_{\rho}$ $\Gamma$-converges to 
	\begin{equation}\label{G0}
		\mathcal{G}_0(u) :=
		\begin{cases*}
			\displaystyle \int_{\Sigma_1} \mathcal{Q} W_0(\nabla_\alpha u) \, dx +\mathcal{H}^2(J_u) & if $u \in \mathcal{A}$, \\
			+ \infty & otherwise in $L^{0}(\Sigma_{1}, \R^{3})$,
		\end{cases*}
	\end{equation}
	where we have defined the set of limit deformations~$\mathcal{A}$ as
	$$
	\mathcal{A}:=\left\{ u \in \gsbvp(\Sigma_1,\R^3)| \ \mbox{$u$ is independent of $x_3$} \right\}.
	$$
	In particular, notice that~$\mathcal{G}_{0}$ is purely $2$-dimensional, as each deformation $u \in \mathcal{A}$ can be identified with its trace on~$\Sigma$, which we still denote by $u \in GSBV^{p}(\Sigma, \R^{3})$ throughout the paper.
	
	Although such convergence result may be rather expected, the strategy necessary to prove it requires a nontrivial adaption of the arguments of~\cite{Hafsa2005TheNM, HafsaMandallenaRelaxation, hafsa2006nonlinear}, as well as the introduction of new tools. As a first step in our proof we show the density in $GSBV^{p}(\Sigma, \R^{3})$ of piecewise affine functions with polyhedral jump set and approximate gradient with maximal rank (see Theorem~\ref{main result 1}). Such result, which can be extended to higher dimensions if $1 < p \leq 2$ (see Remarks~\ref{homom for n magg 2} and~\ref{r:3.4}), generalizes the approximation theorems of~\cite{Cortesani1997StrongAO, Cortesani1999ADR, Philippis2017OnTA}. The crucial point is to perform our approximation for $GSBV^{p}$-functions with polyhedral jump. In such scenario, the key idea, already used in~\cite{Philippis2017OnTA}, is that a domain with smooth cuts can be deformed into a Lipschitz domain through a bi-$W^{1, \infty}$-diffeomorphism arbitrarily close to the identity. This is clearly the best one can hope for, as bi-Lipschitz diffeomorphisms can not exist due to the low regularity of domains with cuts. We further remark that such diffeomorphisms can be made piecewise affine. On the Lipschitz domains we can apply density results in Sobolev spaces to obtain a piecewise affine approximation satisfying a maximum-rank constraint. In particular, since our limit deformations take values in a higher-dimensional space, the results in \cite{gromov1971construction} via discretization allow us to ensure that the constraint is also satisfied by the subgradients at vertices and interfaces. We refer to Theorem~\ref{main result 1} and Lemma~\ref{aff diffeo lemma} for more details.
	
	With the above density result at hand, the crucial step to prove $\Gamma$-convergence is to estimate the $\Gamma$-limsup of~$\rho^{-1} \mathcal{G}_{\rho}$ with the functional
	\begin{equation}\label{intro-G0}
		\mathcal{G}_0^w(u) :=  \int_{\Sigma_1} W_0(\nabla u) \, dx +\mathcal{H}^2(J_u)
	\end{equation}
	defined for functions~$u \in \mathcal{A}$ whose trace on~$\Sigma$ is piecewise affine, has polyhedral and essentially closed jump set, and has maximum-rank subgradients. We refer to~\eqref{e:G0-def} for the precise formulation of~$\mathcal{G}^{w}_{0}$ and to Theorem~\ref{gamma limsup prop} for the $\Gamma$-limsup estimate. To this aim, at first we approximate each deformation in the domain of~$\mathcal{G}^{w}_{0}$ with $C^{1}$-local immersions, but only on regular subdomains of the non-smooth set~$\Sigma \setminus \overline{J_{u}}$. This is shown in Lemma~\ref{trabelsi lemma modified}. 
	With this and a delicate adaption of the arguments of~\cite[Proposition 5.3]{hafsa2006nonlinear}, the desired estimate of the $\Gamma$-limsup by means of~$\mathcal{G}^{w}_{0}$ can be obtained. We remark that in order to define recovery sequences on the whole domain, we again use the diffeomorphisms introduced in Lemma~\ref{aff diffeo lemma}. In turn, this requires some refinements of the constructions in~\cite{Hafsa2005TheNM, hafsa2006nonlinear} (see Lemmas~\ref{lem interchange subset} and~\ref{selection}).
	
	The final part of the $\Gamma$-limit process follows quite closely the arguments of~\cite[Theorem~2.6]{hafsa2006nonlinear}, provided one replaces the density of local immersions with Lemma~\ref{main result 1} (see Theorem~\ref{gamma limit teo} and Proposition~\ref{ben belg lemma}).
	
	We stress the self-containedness of our paper, although some arguments are inspired by \cite{hafsa2006nonlinear} and adapted to a fractured domain. Indeed, many adaptions and careful controls of the involved constants are necessary in the limit procedure, which made a complete proofs presentation preferable. Therefore, no previous knowledge of the results contained in \cite{Hafsa2005TheNM,hafsa2006nonlinear,benbe1996,trabelsi,Trabelsimodeling} is needed.
	
	To the best of our knowledge, this is the first result in the framework of reduced models for brittle membranes in finite elasticity. We believe our result may pave the way to further reduced theories with different geometries, e.g. shells~\cite{Almi2021, CiarletIII}, or constraints, such as incompressibility~\cite{Conti2009, Trabelsimodeling} and weaker growth assumptions than~($H_{5}$).

	\section{Notations and Preliminaries}
	
	\subsection{Notations} For $n, k \in \N$, we denote by $\mathcal{L}^n$ the Lebesgue measure in $\R^n$ and by $\mathcal{H}^k$ the $k$-dimensional Hausdorff measure in $\R^n$. The symbol $\mathbb{M}^{n \times m}$ stands for the space of matrices with $n$ rows and $m$ columns with real coefficients and $I$ is the identity. For every $r > 0$ and every $x \in \R^n$, we denote by $B_r(x)$ the open
	ball in $\R^n$ of radius $r$ and center $x$. Given $\delta>0$ and $U \subset \R^n$ we set $(U)_\delta:=U+B_\delta(0)$. 
	
	Given a matrix $A \in \mathbb{M}^{3 \times 2}$ we denote $A=(A^1|A^2)$ where $A^1$ and $A^2$ are the columns of $A$. We also say that two matrices $A,B \in \mathbb{M}^{3 \times 2}$ are rank one connected or have a rank one connection if $\textnormal{rank}(A-B)=1$. We call $\textnormal{span}(A)$ the spanning of the columns of $A$ as vectors of $\R^3$. Finally, given two vectors $u,v \in \R^3$, we denote by $u \wedge v$ the usual vector product of $u$ and $v$. 

Let $f:\R^n \to \R^m$ with $m,n \geq 1$. For each point $x \in \R^n$ where $f$ is differentiable we call Jacobian of $f$ in $x$ the matrix $\nabla f(x) \in \mathbb{M}^{m \times n}$ such that 
	\footnote[1]{Notice that in the papers \cite{Hafsa2005TheNM, hafsa2006nonlinear}, which we will often refer to, the authors adopt a different convention for the Jacobian of a function, exchanging rows and columns.}
	$$
	[\nabla f(x)]_{i,j} = \frac{\partial f_i}{\partial x_j}(x) \ \ \ \mbox{for $i=1,\dots,m$ and $j=1,\dots n$}. 
	$$

	In the dimension reduction problem, we will need the definition of quasi-convex and rank one envelops. We recall here their definitions.

	\begin{defn}\label{q convex env}
		Let $f: \, \mathbb{M}^{m \times n} \to [0,+\infty]$ be a Borel measurable function. We define the quasiconvex envelope of $f$, denoted as $\mathcal{Q}f$, the function defined as
		\begin{equation*}
			\mathcal{Q}f(\xi)=\inf \left\{ \frac{1}{|B_1|} \int_{B_1} f(\xi+\nabla \varphi(x)) \, dx: \ \varphi \in W_0^{1,\infty}(B_1,\R^m) \right\} \ \ \mbox{for } \xi \in \mathbb{M}^{m \times n},
		\end{equation*}
		where $B_1 \subset \R^n$ is the unit ball.
	\end{defn}

	\begin{rem}\label{rem on qconv env}
		If $f: \, \mathbb{M}^{m \times n} \to [0,+\infty)$ is a Borel measurable function and locally bounded then, by \cite[Theorem 6.9]{DacorognaDirect}, for every $\xi \in \mathbb{M}^{m \times n}$ it holds
		\begin{equation*}
			\mathcal{Q}f(\xi)=\sup \left\{ h(\xi) : \, \mbox{$h: \mathbb{M}^{m \times n} \to \R$, $h \leq f$ and $h$ is quasiconvex} \right\}.
		\end{equation*}
	\end{rem}
	
	\begin{defn}\label{rkone}
		Let $f: \mathbb{M}^{m \times n} \to [0,+\infty]$ be a Borel measurable function. 
		\begin{enumerate}
			\item[(i)] We say that $f$ is rank one convex if for every $\lambda \in (0,1)$ and every $\xi,\xi' \in \mathbb{M}^{m \times n}$ with $\textnormal{rank}(\xi-\xi')=1$,
			$$
			f(\lambda \xi + (1-\lambda) \xi') \leq \lambda f(\xi)+(1-\lambda) f(\xi').
			$$
			\item[(ii)] By the rank one convex envelope of $f$, that we denote by $\mathcal{R}f$, we mean the
			greatest rank one convex function which is less than or equal to $f$.
		\end{enumerate}
	\end{defn}
	
	\subsection{Density Results in Sobolev Spaces}
	
	We begin by recalling some density results in Sobolev spaces (see e.g. \cite{adams2003sobolev}, Chapter 3).

	\begin{defn}[Piecewise affine functions]
		We say that a function $\psi$ on a domain $\Omega \subseteq \R^n$ is piecewise affine if and only if $\psi$ is continuous and there exists a finite family $\left\{ V_i \right\}_{i \in \mathcal{I}}$ of open disjoint subsets of $\Omega$ such that $\mathcal{L}^n(\Omega \setminus \cup_{i \in \mathcal{I}} V_i) = 0$ and for every $i \in \mathcal{I}$, the restriction of $\psi$ to $V_i$ is affine (notice that we can assume each $V_i$ to be the intersection of $\Omega$ with an $n$-dimensional simplex).
		Let us denote 
		\begin{equation*}
			\textnormal{Aff}(\Omega,\R^m):=\left\{ v \in C(\Omega,\R^m): \ \mbox{$v$ is piecewise affine} \right\},
		\end{equation*}
		and
		\begin{equation*}
			\textnormal{Aff}_c(\Omega,\R^m):=\left\{ v \in \Aff(\Omega,\R^m) : \ v \in C_c(\Omega,\R^m) \right\}
		\end{equation*}
	\end{defn}

	\begin{defn}\label{def of triangul}
		Let $\Omega$ be a bounded open subset of $\R^n$. Let $u \in \Aff(\Omega,\R^m)$. We call triangulation of $u$, and we denote it by $\mathcal{T}_u$, the collection of $T_i \cap \Omega$ where $T_i$ is an $n$-dimensional simplex such that $u|_{T_i \cap \Omega}$ is affine. We call diameter of the triangulation $\mathcal{T}$ the quantity $\max \{ \diam(T_i): \ T_i \cap \Omega \in \mathcal{T} \}$. We call vertex of the triangulation each point in $\Omega$ which is a vertex of some $T_i$. 
		
		In particular, if $n=2$, for every element $V_i$ of the triangulation, we call each closed segment of $\overline{\partial T_i \cap \Omega}$ a side of the triangulation. Finally, in this case, we call vertices of the triangulation of $u$ the endpoints of the sides of the triangulation.
	\end{defn}
	
	\begin{thm}\label{cont density}
		If $ \Omega \subset \R^n$ is an open set with Lipschitz boundary, then the set of restrictions to $\Omega$ of functions in $C_c^\infty(\R^n,\R^m)$ are dense in $W^{1,p}(\Omega,\R^m)$ for $1 \leq p<\infty$.
	\end{thm}
	
	\begin{rem}\label{pwa density}
		It follows by Theorem \ref{cont density} that restrictions to $\Omega$ of functions in $\textnormal{Aff}_c(\R^n,\R^m)$ is dense in $W^{1,p}(\Omega,\R^m)$ when $\Omega$ is an open set with Lipschitz boundary.
	\end{rem}

	In order to obtain an approximation with locally injective functions we need the following definitions (see \cite[Section 2.6]{Clarke}).
	
	\begin{defn}[Clarke subdifferential]
		Let $f : \ \Omega \subseteq \R^n \to \R^m$ be a locally Lipschitz function defined on an open set $\Omega$ of $\R^n$. For every $x \in \overline{\Omega}$ we define the Clarke subdifferential $\partial f(x)$ of $f$ in $x$ as
 		$$
		 \partial f(x):=\textnormal{conv}\left(\left\{ \lim_{k \to +\infty} \nabla f(x_k): \ x_k \to x, \ \nabla f(x_k) \ \mbox{is well defined} \right\}\right),
		$$
		where conv denotes the convex hull.
	\end{defn}

	\begin{defn}
		Let $\Omega \subset \R^n$ be a bounded open set and let $n \leq m$. If $m>n$, we define
		\begin{align*}
			\Aff^*(\Omega,\R^m):=\left\{ v \in \Aff(\Omega,\R^m): \ \min \{\det(A^TA): \ A \in \partial v(x), \ x \in  \Omega \}>0 \right\}.
		\end{align*}
		Otherwise, if $n=m$,
		\begin{align*}
			\Aff^*(\Omega,\R^n):=\left\{ v \in \Aff(\Omega,\R^n): \ \det(\nabla v(x))>0 \ \ \mbox{for a.e. $x \in \Omega$} \right\}.
		\end{align*}
	\end{defn}

	\begin{rem}
		If $v \in \Aff^*(\Omega,\R^m)$ and $m>n$, then $v$ is locally injective in $\Omega$. Indeed, let $x_0 \in \Omega$ and take $\delta>0$ small such that $B_\delta(x_0)\subset \Omega$ and $\partial v(x) \subseteq \partial v(x_0)$ for every $x \in B_\delta(x_0)$. Given $y,z \in B_\delta(x_0)$ the segment connecting $y$ and $z$ intersects at most $K$ triangles of $\mathcal{T}_v$ denoted by $T_1,\dots,T_K$. Let us fix $x_k \in T_k$ for every $k=1,\dots,K$. We have 
		\begin{align*}
			v(z)-v(y) &= \int_0^{|z-y|} \nabla v \left( y+t\frac{z-y}{|z-y|} \right) \frac{z-y}{|z-y|} \, dt \\ & = \left(\sum_{k=1}^K \lambda_k \nabla v(x_k) \right) (z-y),
		\end{align*} 
		where $\lambda_k \geq 0$ and $\sum_k \lambda_k=1$. Therefore, $v(z)-v(y) =A(z-y)$ where $A \in \partial v(x_0)$. Since $v \in \Aff^*(\Omega,\R^m)$, this implies  $|v(z)-v(y)|\geq c|z-y|$ for some $c>0$ for every $y,z \in B_\delta(x_0)$.
	\end{rem}
	
	We present a density result due to Gromov and Eliashberg (see  \cite[Section 2.2.1]{gromov1986partial} and \cite{gromov1971construction}). A similar (actually stronger) result is proven, for the case $n=2$ and $m=3$, by Conti and Dolzmann in \cite[Proposition 4.1]{conti-dolzman}.
	
	\begin{thm}\label{thm Gromov}
		Let $1 \leq p<\infty$ and let $1 \leq n < m$ be two integers and let $M$ be a compact $n$-dimensional manifold which can be immersed in $\R^m$. Then, for each $\psi \in C^1(M;\R^m)$ there exists a sequence $\psi_j \in C^1(M,\R^m)$ such that $\psi_j$ is an immersion for every $j \geq 1$ and $\psi_j \to \psi$ in $W^{1,p}(M,\R^m)$.
	\end{thm}

$C^1$-immersions can be approximated by locally injective functions in $\Aff^*$ by discretization, as we state below.
	
	\begin{prop}\label{affine inj}
		Let $n \leq m$, $R \subset \R^n$ be an open $n$-dimensional rectangle and $u \in C^1(\overline{R},\R^m)$. Then, the following facts hold: 
		\begin{itemize}
			\item[\textnormal{(i)}] for every $\sigma>0$ there exists $u_\sigma \in \Aff(R,\R^m)$ such that $\Vert u-u_\sigma \Vert_{W^{1,\infty}} \leq \sigma$ and $\partial u_\sigma(x) \subseteq B_\sigma(\nabla u(x))$ for every $x \in R$;
			\item[\textnormal{(ii)}] if $\textnormal{rank}(\nabla u)=n $ on $\overline{R}$, then $u_\sigma \in \Aff^*(R,\R^m)$ for $\sigma$ sufficiently small.
		\end{itemize}
	\end{prop}
       
	\begin{proof}
		Item (i) follows by a standard discretization argument, using the fact that $u \in C^1(\overline{R},\R^m)$.
		
		As for item (ii), let us fix $\eta>0$ such that $\det((\nabla u(x))^T(\nabla u(x)))\geq \eta$. It follows by (i) and continuity of determinant that for $\sigma$ small we have
		$$
		\det(A^TA) \geq \frac{\eta}{2} \ \ \ \mbox{for every $A \in \partial u_\sigma(x)$ and every $x \in R$}.
		$$
		
	\end{proof}

		Using Theorem \ref{thm Gromov} (if $m>n$) and Proposition \ref{affine inj} we can actually show a stronger density result than the one stated in Remark \ref{pwa density} which involves locally injective functions.
		
		\begin{cor}
			Let $1 \leq p <\infty$, $n \leq m$ and $\Omega \subset \R^n$ be an open bounded set with Lipschitz boundary. For every $u \in W^{1,p}(\Omega,\R^m)$ there exist a rectangle $R \Supset \Omega$ in $\R^n$ and a sequence $u_j \in \Aff^*(R,\R^m)$ such that $u_j|_\Omega \to u$ in $W^{1,p}(\Omega,\R^m)$.
		\end{cor}
		
		\begin{proof} 
			
		If $n=m$, by Remark \ref{pwa density} we find $v \in \Aff_c(\R^n,\R^m)$ such that the restriction of $v$ in $\Omega$ approximates $u$ in the $W^{1,p}(\Omega,\R^m)$ norm. Consider $R \Supset \Omega$ a rectangle in $\R^n$. We can use a classical topological argument applied to the restriction of $v$ in $R$ (see e.g. \cite[Proposition 3.1.6]{trabelsi}) in order to obtain $w \in \Aff^*(R,\R^m)$. Thus allowing us to conclude in the case $m=n$.
		
		If $m>n$, we may first consider $u \in C_c^\infty(\R^n,\R^m)$. Consider $R \Supset \Omega$ a rectangle in $\R^n$ and apply Theorem \ref{thm Gromov} to the restriction of $u$ to $\overline{R}$ to construct a sequence of immersions $\phi_j \in C^1(\overline{R},\R^m)$ such that $\phi_j|_\Omega \to u$ in $W^{1,p}(\Omega,\R^m)$. To conclude it is enough to apply Proposition \ref{affine inj} to each $\phi_j$. The general case follows by density using Theorem \ref{cont density}.
		
		\end{proof}

	\subsection{Density in GSBV}
	
	We now briefly recall some basic definitions and results in the space $\gsbv$. We refer also to \cite[Section 4.5]{Ambrosio2000FunctionsOB} for more details on this topic.
	
	For $\Omega \subseteq \R^n$ open, let $x \in \Omega$ and $v: \Omega \to \R^m$ be an $\mathcal{L}^n$-measurable function such that 
	$$
	\limsup_{r \searrow 0} \frac{\mathcal{L}^n(\Omega \cap B(x,r))}{r^n}>0,
	$$
	we say that $a \in \R^m$ is the approximate limit of $v$ at $x$ if
	$$
	\lim_{r \searrow 0} \frac{\mathcal{L}^n(\Omega \cap B(x,r) \cap \left\{ |a-x|>\varepsilon \right\})}{r^n}=0 \ \ \ \mbox{for every $\varepsilon>0$}. 
	$$
	In that case we write 
	$$
	\textnormal{ap-}\lim_{y \to x} v(y)=a.
	$$
	We say that $x \in \Omega$ is an approximate jump point of $v$, and we write $x \in J_v$, if there exists $a, b \in \R^m$ with $a \neq b$ and $\nu \in \mathbb{S}^{n-1}$ such that
	$$
	\underset{\underset{(y-x) \cdot \nu>0}{y \to x}}{\textnormal{ap-}\lim} v(y)=a \ \ \ \ \mbox{and} \ \ \ \ \underset{\underset{(y-x) \cdot \nu<0}{y \to x}}{\textnormal{ap-}\lim} v(y)=b.
	$$
	In particular, for every $x \in J_v$ the triple $(a, b, \nu)$ is uniquely determined up to a change of sign of $\nu$ and a permutation of $a$ and $b$. We indicate such triple by $(v^+(x), v^-(x), \nu_v(x))$. The jump of $v$ at $x \in J_v$ is defined as $[v](x)=v^+(x)-v^-(x)$. 
	
	The space $\textnormal{BV}(\Omega,\R^m)$ of functions of bounded variation is the set of $u \in L^1(\Omega;\R^m)$ whose distributional gradient $Du$ is a bounded Radon measure on $\Omega$ with values in $\mathbb{M}^{m \times n}$. Given $u \in \textnormal{BV}(\Omega,\R^m)$ we can write $Du=D^a u+ D^s u$, where $D^a u$ is absolutely continuous and $D^s u$ is singular w.r.t. $\mathcal{L}^n$. The set $J_u$ is countably rectifiable and has approximate unit normal vector $\nu_u$, while the density $\nabla u \in L^1(\Omega,\mathbb{M}^{m \times n})$ of $D^a u$ w.r.t. $\mathcal{L}^n$ coincides a.e. in $\Omega$ with the approximate gradient of $u$. That is, for a.e. $x \in \Omega$ it holds
	$$
	\underset{y \to x}{\textnormal{ap-}\lim} \frac{u(y)-u(x)-\nabla u(x)\cdot (y-x)}{|x-y|}=0.
	$$
	
	The space $\sbv(\Omega,\R^m)$ of special functions of bounded variation is defined as
	the set of all $u \in \textnormal{BV}(\Omega,\R^m)$ such that $|D^s u|(\Omega \setminus J_u)=0$. Moreover, we denote by $\sbv_{\textnormal{loc}}(\Omega,\R^m)$ the space of functions belonging to $\sbv(U,\R^m)$ for every $U \Subset \Omega$. For $p \in [1,+\infty)$, $\sbvp(\Omega,\R^m)$ stands for the set of functions $u \in \sbv(\Omega,\R^m)$, with approximate gradient $\nabla u \in L^p(\Omega,\mathbb{M}^{m \times n})$ and $\mathcal{H}^{n-1}(J_u)<+\infty$.
	
	We say that $u \in \gsbv(\Omega,\R^m)$ if $\varphi(u) \in \sbv_{\textnormal{loc}}(\Omega,\R^m)$ for every $\varphi \in C^1(\R^m,\R^m)$ whose gradient has compact support. Also for $u \in \gsbv(\Omega,\R^m)$ the approximate gradient $\nabla u$ exists $\mathcal{L}^n$-a.e. in $\Omega$ and the jump set $J_u$ is countably $\mathcal{H}^{N-1}$-rectifiable with approximate unit normal vector $\nu_u$. Finally, for $p \in [1,+\infty)$, we define $\gsbvp(\Omega,\R^m)$ as the set of functions $u \in \gsbv(\Omega,\R^m)$, with approximate gradient $\nabla u \in L^p(\Omega,\mathbb{M}^{m \times n})$ and $\mathcal{H}^{n-1}(J_u)<+\infty$. \\
	
	\begin{defn}
		We denote by $\mathcal{W}(\Omega;\R^m)$ the space of all functions $u \in \sbv (\Omega;\R^m)$ with the following properties:
		\begin{enumerate}
			\item[(i)] $u \in W^{1,\infty}(\Omega \setminus \overline{J_{u}},\R^m)$,
			\item[(ii)] $J_u$ is essentially closed, i.e., $\mathcal{H}^{n-1}(\overline{J_u} \setminus J_u) = 0$,
			\item[(iii)] $J_u$ is composed by the intersection of $\Omega$ with the finite union of ($n-1$)-dimensional polyhedra.
		\end{enumerate}
	\end{defn}

	Let us state a classic approximation result for $\gsbvp$ functions, obtained in \cite{Cortesani1997StrongAO} and subsequentially refined in \cite{Cortesani1999ADR}.
	
	\begin{thm}[Cortesani-Toader] \label{ct_approx}
		Let $\Omega \subset \R^n$ be an open bounded set with Lipschitz boundary. Let $u \in \gsbvp(\Omega,\R^m)$ with $p>1$. Then, there exists a sequence of functions $u_j \in \mathcal{W}(\Omega,\R^m)\cap C^\infty(\Omega \setminus \overline{J_{u_j}},\R^m)$ such that $u_j \to u$ in measure, $\nabla u_j \to \nabla u$ in $L^p(\Omega,\mathbb{M}^{m \times n})$ and 
		\begin{equation} \label{limsup}
			\limsup_{j \to \infty} \int_{J_{u_j}\cap \overline{A}} g(x,u_j^+,u_j^-,\nu_{u_j}) \,d\mathcal{H}^{n-1} \leq \int_{J_{u}\cap \overline{A}} g(x,u^+,u^-,\nu_{u}) \,d\mathcal{H}^{n-1}
		\end{equation}
		for every $A \subset \subset \Omega$ open and every upper semicontinuous function $g: \Omega \times \R^m \times \R^m \times \mathbb{S}^{n-1} \to [0,\infty)$ such that $g(x,a,b,\nu)=g(x,a,b,-\nu)$ for all $x \in \Omega$, $a,  b \in \R^m$, $\nu \in \mathbb{S}^{n-1}$. If $u \in L^\infty(\Omega)$, we may further assume that $\Vert u_j \Vert_{L^\infty} \leq \Vert u \Vert_{L^\infty}$ for every $j \geq 1$. 
	\end{thm}
	
	\begin{rem}\label{ct remark}
		By \cite[Remark 3.2]{Cortesani1999ADR}, if $g$ is locally bounded near $\partial \Omega$ we can choose the sequence $\left\{ u_j \right\}_j$ in such a way that (\ref{limsup}) holds for every $A \subseteq \Omega$, in this case $\overline{A}$ must be replaced by the relative closure of $A$ in $\Omega$.
	\end{rem}
	
	\begin{rem}\label{cap simplexes}
		Using \cite[Lemma 5.2]{Philippis2017OnTA} we can always assume that $J_{u_j} \Subset  \Omega$ for each $j \geq 1$, where $\left\{u_j \right\}_j$ is the sequence of approximants given in Theorem \ref{ct_approx}.
		Moreover, if $1<p \leq 2$, the capacitary argument of \cite[Corollary 3.11]{Cortesani1997StrongAO} holds. Thus, we can additionally assume that the jump set of the approximants is composed by the finite union of disjoint ($n-1$)-dimensional simplexes.
	\end{rem}

	We finally present a result concerning the integral representation of functionals defined on $\gsbvp$. More details can be found in \cite[Theorem 3.5 and Theorem 3.8]{Cagnetti2019}.
	
	\begin{thm}\label{relaxation gsbvp}
		Let $\Omega \subset \R^n$ open, let $f: \mathbb{M}^{m \times n} \to \R_+$ be a Borel function, and assume that there exist $c,C>0$ and $p \in (1,\infty)$ such that 
		$$c|M|^p-\frac{1}{c} \leq f(M) \leq C(1+|M|^p) \ \ \ \mbox{for all $M \in \mathbb{M}^{m \times n}$.}$$
		Then the lower semicontinuous envelope of the functional 
		$$
		\int_{\Sigma} f(\nabla u) \, dx + \mathcal{H}^{n-1}(J_u) \ \ \ u \in \gsbvp(\Omega,\R^m)
		$$
		with respect to the convergence in measure is given by
		$$
		\int_{\Omega} \mathcal{Q}f(\nabla u) \, dx + \mathcal{H}^{n-1}(J_u) \ \ \ u \in \gsbvp(\Omega,\R^m),
		$$
		where $\mathcal{Q}f$ is the quasi convex envelope of $f$.
	\end{thm}

	\noindent From now on we will consider mainly the case $n=2$. 
	
	\begin{defn}\label{W hat def}
		Let $\Omega \subset \R^2$ be an open bounded set with Lipschitz boundary. 
		We denote by $\widehat{\mathcal{W}}(\Omega;\R^m)$ the space of all functions $u \in \mathcal{W}(\Omega,\R^m)$ with the following property: each connected component of the jump set of $u$ is either a segment or the union of two segments intersecting only at one endpoint and whose convex hull is contained in $\Omega$.
	\end{defn}
	
	In Theorem \ref{ct_approx}, even if $p>2$, we can take the sequence of approximants such that it is contained in $\widehat{\mathcal{W}}(\Omega,\R^m)$. Indeed, the following Lemma holds.
	
	\begin{lem}\label{segments lemma}
		Let $\Omega \subset \R^2$ be an open bounded Lipschitz domain, $m \geq 1$ and $u \in \mathcal{W}(\Omega,\R^m) \cap C^1(\Omega \setminus \overline{J_u},\R^m)$ with a polyhedral jump set $J_u \Subset \Omega$. For every $\varepsilon >0$ there exists $u_\varepsilon \in \widehat{\mathcal{W}}(\Omega,\R^m) \cap C^\infty(\Omega \setminus \overline{J_u},\R^m) $ such that
		\begin{equation*}
			\Vert u - u_\varepsilon \Vert_{L^1} < \varepsilon, \ \ \ \ \Vert \nabla u - \nabla u_\varepsilon \Vert_{L^p}<\varepsilon, \ \ \ \ \mathcal{H}^1(J_u \Delta J_{u_\varepsilon}) <\varepsilon.
		\end{equation*}
	\end{lem} 
	
	The above property is essentially proved in \cite[Lemma 5.2]{Philippis2017OnTA} which however has a different statement in terms of $C^1$ manifolds. Actually, the proof is essentially based on an argument for polyhedral sets which also allows one to deduce the statement in Lemma \ref{segments lemma}. The slight adaptions which are required to this aim are sketched in the Appendix for completeness.

	\section{Approximation Results with a Maximal Rank Condition}
	
	The goal of this section is proving an approximation result for $\gsbvp$-functions in the spirit of Theorem \ref{ct_approx} with a maximal rank constraint on the Jacobian of the approximating sequence.

	\begin{thm} \label{main result 1}
		Let $\Omega \subset \R^2$ be an open bounded set with Lipschitz boundary and $1<p<\infty$. Let $u \in \gsbvp(\Omega,\R^m)$ with $m \geq 2$. Then, there exists a sequence of functions $u_j \in \Aff^*(\Omega \setminus \overline{J_{u_j}},\R^m) \cap \widehat{\mathcal{W}}(\Omega,\R^m)$ such that:
		\begin{enumerate}
			\item[\textnormal{(i)}] $u_j \to u$ in measure on $\Omega$;
			\item[\textnormal{(ii)}] $\nabla u_j \to \nabla u$ in $L^p(\Omega,\mathbb{M}^{m \times n})$;
			\item[\textnormal{(iii)}] $J_{u_j} \Subset \Omega$ for every $j \geq 1$ and $ \limsup_{j \to \infty} \mathcal{H}^1(J_{u_j}) \leq \mathcal{H}^1(J_{u}) $.
		\end{enumerate}
	\end{thm}

	To prove Theorem \ref{main result 1}, we need the following technical Lemma which deals with the construction of an homeomorphism.

	\begin{lem}\label{aff diffeo lemma}
		Let $\Pi \subset \R^2$ be a segment or the union of two segments intersecting only at one endpoint. Then, for $\delta>0$ small enough, there exist $\Delta \subset \R^2$ depending on $\delta$ and a piecewise affine homeomorphism $\Phi_\delta: \R^2 \setminus \Pi \to \R^2 \setminus \Delta$, such that the following properties hold: 
		\begin{enumerate}
			\item[\textnormal{(i)}] $\R^2 \setminus \Delta$ is a Lipschitz domain;
			\item[\textnormal{(ii)}] $\Phi_\delta(x)=x$ for every $x \in \R^2 \setminus (\Pi)_\delta$;
			\item[\textnormal{(iii)}] $\Vert \Phi_{\delta}-\textnormal{Id} \Vert_{W^{1,\infty}(\R^2 \setminus \Pi)} \to 0$ as $\delta \to 0$.
		\end{enumerate}
	\end{lem}
	
	\begin{proof}
		We divide the proof into two cases.
		
		\noindent \textbf{Case 1:} We first assume that $\Pi$ is a segment. Without loss of generality, we can suppose that $\Pi \subset \left\{ x_2=0 \right\}$.
		
		We claim that there exists a continuous piecewise affine function $f : \R \to \R_{\geq 0}$ such that for every $\delta>0$: 
		\begin{enumerate}
			\item[\textnormal{(1)}] $f(x) > 0$ for every $x \in \Int(\Pi)$, $f(x) = 0$ for every $x \in \R \setminus \Int(\Pi)$ and $\max(f)=1$ on $\R$,
			\item[\textnormal{(2)}] Setting $\Delta:= \left\{x \in \Pi \times \R \ : \ x_2 \leq \delta f(x_1)\right\}$, the open set $\R^2 \setminus \Delta$ is a Lipschitz domain.
		\end{enumerate} 
		To prove the claim, let $p^1$ and $p^2$ be the endpoints of $\Pi$ and for $z=(z_1,0) \in \Int(\Pi)$, let $p^0=(z_1,1)$. On $\Pi$ we define $f$ as the piecewise affine function whose graph is given by the two segments connecting $p^1$ with $p^0$ and $p^0$ with $p^2$. We further set $f=0$ on $\R \setminus \Pi$. By construction, $f$ satisfies the requirements. Observe that the set $\Delta$ is a triangle for every $\delta >0$. Therefore, $\R^2 \setminus \Delta$ is a Lipschitz domain.

		Set $f_\delta(x):=\delta f(x)$.
		We define $\Phi_\delta: \R^2 \setminus \Pi \to \R^2 \setminus \Delta$ as follows:
		\begin{equation}\label{homeo defn}
			\Phi_\delta(x):=
			\begin{cases*}
				\displaystyle \left(x_1,\frac{x_2}{1+\delta} + \frac{\delta f_\delta(x_1)}{1+\delta}\right) & $x_1 \in \Pi$ and $0<x_2<f_\delta(x_1)$, \\
				x & otherwise in $\R^2 \setminus \Pi$.
			\end{cases*}
		\end{equation}
		Notice that $\Phi_\delta$ is a continuous bijection from $\R^2 \setminus \Pi$ and $\R^2 \setminus \Delta$ and $\Phi_\delta(x)=x$ outside $(\Pi)_\delta$. 
		It either holds
		$$
		D\Phi_\delta = I \ \ \ \mbox{or} \ \ \ 
		D\Phi_\delta(x)=
		\left( \begin{array}{cc}
			1 & 0 \\
			\displaystyle  \frac{\delta f'_\delta(x_1)}{1+\delta}   & \displaystyle \frac{1}{1+\delta}
		\end{array}   \right).
		$$
		Since $f_\delta$ is piecewise affine and continuous, we have that $\Phi_\delta$ is piecewise affine on $\R^2 \setminus \Pi$. Moreover, by a direct computation we get
		\begin{equation*}
			\Phi_\delta^{-1}(x)=
			\begin{cases*}
				(x_1,(1+\delta)x_2 -\delta f_\delta(x_1)) & $ x_1 \in \Pi$ and $\displaystyle \frac{\delta f_\delta(x_1)}{1+\delta}<x_2<f_\delta(x_1)$, \\
				x & otherwise in $\R^2 \setminus \Delta$.
			\end{cases*}
		\end{equation*}
		As for the differential of $\Phi_\delta^{-1}$ it either holds 
		$$
		D\Phi_\delta^{-1}= I \ \ \ \mbox{or} \ \ \ 
		D\Phi_\delta^{-1}(x)=
		\left( \begin{array}{cc}
			1 & 0 \\
			- \delta f'_\delta(x_1)  & 1+\delta
		\end{array}   \right).
		$$
		Finally, since $f_\delta$ and $f_\delta'$ are uniformly bounded with respect to $\delta$, we have that
		$$\lim_{\delta \to 0} \Vert \Phi_\delta-\textnormal{Id} \Vert_{W^{1,\infty}(\R^2 \setminus \Pi)} = 0.$$
		
		\textbf{Case 2:} Now we assume that $\Pi$ is the union of two segments intersecting only at one endpoint. We call $\gamma$ the segment connecting the endpoints of $\Pi$. As in the previous case, it is not restrictive to assume $\gamma \subset \{x_2=0\}$. We further define $T:=\textnormal{conv}(\Pi)$, and observe that $\partial T=\Pi \cup \gamma$. 
		
		\noindent \textbf{Step 1:} We claim that there exists a piecewise affine homeomorphism $\Psi:\R^2  \to \R^2 $ such that $\Psi(x)=x$ out of an arbitrarily small neighborhood $U$ of $T$, $\Psi(\Pi)=\gamma$ and $\Psi$ is bi-Lipschitz of constant $M=M(U)$.
		
		To construct $\Psi$, let $T^-$ be a triangle contained in the half-plane opposite to $T$, having a side coinciding with $\gamma$ and such that the height of $T^-$ is arbitrarily small. Consider two triangles $\widehat{T}$ and $\widehat{T}^-$ with basis $\gamma$, with the property $T \subset \widehat{T}$, $T^- \subset \widehat{T}^-$ in such a way that $|(\widehat{T} \setminus T) \cup (\widehat{T}^- \setminus T^-)|$ is arbitrarily small. Finally suppose that all their free vertices are aligned on a line perpendicular to $\gamma$. 
		We can construct $\Psi$ piecewise affine, as depicted in Figure \ref{f:2}, such that $\Psi(\widehat{T} \setminus T)=\widehat{T} \cup T$, $\Psi(T)=T^-$, $\Psi(\widehat{T}^- \cup T^-)=\widehat{T}^- \setminus T^-$ and $\Psi=\textnormal{Id}$ elsewhere. Notice that the triangulation of the set $\left\{ \Psi(x) \neq x \right\}$ is composed by the eight triangles depicted in Figure \ref{f:2}.
		
		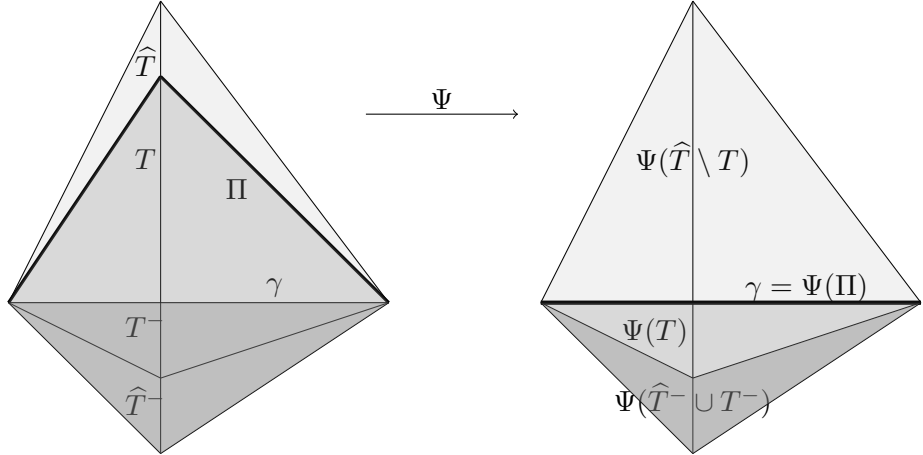
\begin{figure}[h!]
			\begin{tikzpicture}
				
				\draw[thin] (0,0) -- (5,0);
				\draw[very thick] (0,0) -- (2,3);
				\draw[very thick] (2,3) -- (5,0);
				\draw[thin] (0,0) -- (2,4);
				\draw[thin] (2,4) -- (5,0);
				\draw[thin] (0,0) -- (2,-1);
				\draw[thin] (2,-1) -- (5,0);
				\draw[thin] (0,0) -- (2,-2);
				\draw[thin] (2,-2) -- (5,0);
				\draw[thin] (2,4) -- (2,-2);
				\node at (1.8,1.9) {$T$};
				\node at (1.8,3.2) {$\widehat{T}$};
				\node at (1.8,-0.3) {$T^-$};
				\node at (1.8,-1.3) {$\widehat{T}^-$};
				\node at (3,1.5) {$\Pi$};
				\node at (3.5,0.2) {$\gamma$};
				\draw[ultra thick] (7,0) -- (12,0);
				\draw[thin] (7,0) -- (9,4);
				\draw[thin] (9,4) -- (12,0);
				\draw[thin] (7,0) -- (9,-1);
				\draw[thin] (9,-1) -- (12,0);
				\draw[thin] (7,0) -- (9,-2);
				\draw[thin] (9,-2) -- (12,0);
				\draw[thin] (9,4) -- (9,-2);
				\draw[->] (4.7,2.5) -- (6.7,2.5);
				\node at (5.7,2.7) {$\Psi$};
				\node at (9,1.9) {$\Psi(\widehat{T} \setminus T)$};
				\node at (8.5,-0.4) {$\Psi(T)$};
				\node at (9.0,-1.3) {$\Psi(\widehat{T}^- \cup T^-)$};
				
				\node at (10.5,0.2) {$\gamma=\Psi(\Pi)$};
				\fill[fill= gray, fill opacity=0.1] (0,0) -- (2,4) -- (5,0) -- (2,3) -- cycle;
				\fill[fill= gray, fill opacity=0.5] (0,0) -- (2,-2) -- (5,0) -- cycle;
				\fill[fill= gray, fill opacity=0.3] (0,0) -- (2,3) -- (5,0) -- cycle;
				\fill[fill= gray, fill opacity=0.1] (7,0) -- (9,4) -- (12,0) -- cycle;
				\fill[fill= gray, fill opacity=0.5] (7,0) -- (9,-1) -- (12,0) -- (9,-2) -- cycle;
				\fill[fill= gray, fill opacity=0.3] (7,0) -- (9,-1) -- (12,0) -- cycle;
			\end{tikzpicture}
			\caption{The way $\Psi$ acts in $\left\{ \Psi \neq \textnormal{Id} \right\}$. Notice that $\Psi(C)=C$, $\Psi(D)=E$, $\Psi(E)=F$, $\Psi(G)=G$ and $\Psi(F)$ is a point contained in the line joining $\Psi(E)$ and $G$.} \label{f:2}
		\end{figure}
		Furthermore, by construction $\Psi (x)=x$ for every $x \notin \widehat{T} \cup \widehat{T}^-$, which is contained in a small neighborhood $U$ of $T$, and
		\begin{equation}\label{boundedness of DPsi}
			\| D\Psi \Vert_{W^{1,\infty}(\R^2)} + \Vert D\Psi^{-1} \Vert_{W^{1,\infty}(\R^2 )} \leq M=M(U).
		\end{equation}
		Finally, it follows by construction that
		\begin{equation} \label{D psi -1}
			D\Psi^{-1}(x_1,x_2)=D\Psi^{-1}(x_1,\tilde{x}_2) \ \ \ \mbox{for every $(x_1,x_2), \ (x_1,\tilde{x}_2) \in T$}.
		\end{equation}

		\noindent \textbf{Step 2:} Let $\Psi$, $M$ and $\gamma$ be given by the previous step. Fix $\delta \in (0,1/2)$, as in Case 1 we construct $\Theta_\delta:\R^2 \setminus \gamma \to \Omega$ such that the image set $\Omega$ is a Lipschitz domain (depending on $\delta$) and $\Theta_\delta(x)=x$ for every $x \in \R^2 \setminus (\gamma)_{\delta/M}$. Observe that by the construction given in Case 1, see (\ref{homeo defn}), we may assume that 
		$$
		\langle \Theta_\delta (x), e_1 \rangle=x_1 \ \ \mbox{for every $x \in \R^2 \setminus \gamma$}.
		$$ 
		Up to taking $\delta$ small enough, $\left\{ \Theta_\delta \neq \textnormal{Id} \right\} \subseteq T$, where $T$ is given in the previous step. Combining these properties with (\ref{D psi -1}), we infer that
		\begin{equation}\label{D psi -1 2}
			D\Psi^{-1}(\Theta_\delta(y))=D\Psi^{-1}(y) \ \ \ \mbox{for a.e. $y \in \R^2 \setminus \gamma$}.
		\end{equation}
		
		We are now in a position to define $\Phi_\delta$. Let $\delta>0$ small enough such that (\ref{D psi -1 2}) holds. Set $\Delta:=\R^2 \setminus \Psi^{-1}(\Omega)$, define $\Phi_\delta:\R^2 \setminus \Pi \to \R^2 \setminus \Delta$ as $\Phi_\delta:=\Psi^{-1} \circ \Theta_\delta \circ \Psi$. 
		Observe that $\Phi_\delta$ is piecewise affine and that it is the identity in $\R^2 \setminus (\Pi)_\delta$.
		Indeed, if $x \notin \Pi_\delta$, then $\Psi(x) \notin (\gamma)_{\delta/M}$. Hence $\Theta_\delta(\Psi(x))=\Psi(x)$ and $\Phi_\delta(x)=x$ and this proves property (ii).
		
		Observe that $\Theta_\delta(\R^2 \setminus \gamma)$ is the complementary set of a triangle. Hence $\Delta$, which is the image of such a triangle through the piecewise affine bi-Lipschitz homeomorphism $\Psi$, is a polygon without self intersections. It follows that the complementary set $\R^2 \setminus \Delta$ is a Lipschitz open set. This is property (i).
		
		We are left to prove property (iii). From (\ref{D psi -1 2}) we deduce that for almost every $x \in \R^2 \setminus \Pi$,
		\begin{align*}
			|D\Phi_\delta(x)-x| & = |D\Psi^{-1} (\Theta_\delta(\Psi(x)))D\Theta_\delta(\Psi(x))D\Psi(x)-x| \\ & =|D\Psi^{-1} (\Psi(x))D\Theta_\delta(\Psi(x))D\Psi(x)-x|.
		\end{align*}
		Thus, it follows from (\ref{boundedness of DPsi}) that $\Vert D\Phi_\delta -I \Vert_{L^\infty} \leq M^2 \Vert D\Theta_\delta -I \Vert_{L^\infty}$. Since $D\Theta_\delta \to I$ uniformly, we infer that $D\Phi_\delta \to I $ as $\delta \to 0$ uniformly in $\R^2 \setminus \Pi$. This implies $\Vert \Phi_\delta -\textnormal{Id} \Vert_{W^{1,\infty}(\R^2 \setminus \Pi)} \to 0$ as $\delta \to 0$ and concludes the proof of (iii).
	\end{proof}
	
	\begin{rem}\label{homom for n magg 2}
		We notice en passant that the same construction can be generalized for $n \geq 3$ when $\Pi$ is an ($n-1$)-dimensional simplex. In this case, the set $\Delta$ is an $n$-dimensional simplex with base $\Pi$ and the homeomorphism can be constructed analogously to Case 1. Finally, since $\Delta$ is a convex polyhedron we have that $\R^n \setminus \Delta$ is a Lipschitz domain by \cite[Theorem 1.2.2.3]{Grisvard}.
	\end{rem}

	We are now in a position to prove Theorem \ref{main result 1}
	
	\begin{proof}[Proof of Theorem \ref{main result 1}]
		By a double sequence argument and Theorem \ref{ct_approx} it suffices to prove that given $u \in \widehat{\mathcal{W}}(\Omega,\R^m) \cap C^\infty(\Omega \setminus \overline{J_{u}},\R^m)$, for all $\varepsilon >0$, there exists a function $\tilde{u}$ with the desired properties and such that 
		\begin{equation*}
			\Vert u - \tilde{u} \Vert_{W^{1,p}(\Omega \setminus \overline{J_u})} \leq 12\varepsilon, \ \ \ \ \ \ \ J_{\tilde{u}} \subseteq J_{u}.
		\end{equation*}

		%
		
		For $i=1,\dots,K$ denote the connected components of $\overline{J_u}$ with $\Pi_i$. Correspondingly, we fix $A_i \Subset \Omega$ an open smooth neighborhood of $\Pi_i$. We can assume that the sets $A_i$ are pairwise disjoint.
		
		Using Lemma \ref{aff diffeo lemma} we construct $W^{1,\infty}$-piecewise affine homeomorphisms $\Phi_i: A_i \setminus \Pi_i \to A_i \setminus D_i$, such that $A_i \setminus D_i$ is a Lipschitz domain and $\Phi_i(x)=x$ outside a neighborhood of $\Pi_i$ compactly contained in $A_i$. 
		Moreover, using (iii) of Lemma \ref{aff diffeo lemma} we can assume that for every $i=1,\dots,K$,
		\begin{align}\label{phii_bound}
			\begin{split}
				& \Vert D\Phi \Vert_{L^\infty(A_i \setminus \Pi_i)}+\Vert D\Phi^{-1} \Vert_{L^\infty(A_i \setminus D_i)} \leq 3 \\ & \det(M) \geq \frac{1}{2} \ \ \mbox{for every $M \in \partial \Phi_i(x)$ and for every $x \in  A_i \setminus \Pi_i$}.
			\end{split}
		\end{align}
		Let us set $V:=\Omega \setminus\cup_{i=1}^K D_i $. We define $\Phi: \ \Omega \setminus \overline{J_{u}} \to V$ as
		\begin{equation*}
			\Phi(x) := 
			\begin{cases*}
				x & if $x \in \Omega \setminus \bigcup_{i=1}^K A_i$ \\
				\Phi_i(x) & if $ x \in A_i \setminus \Pi_i$, for $i=1,\dots,K$.
			\end{cases*}
		\end{equation*}
		Moreover, by (\ref{phii_bound}) and using the fact that the $A_i$'s are pairwise disjoint, we have that
		\begin{align}\label{phi_bound}
			\begin{split}
			& \Vert D\Phi \Vert_{L^\infty(\Omega)}+\Vert D\Phi^{-1} \Vert_{L^\infty(V)} \leq 3 \\ & \det(M) \geq \frac{1}{2} \ \ \mbox{for every $M \in \partial \Phi(x)$ and for every $x \in  \Omega \setminus \overline{J_{u}}$}.
			\end{split}
		\end{align}


		For every $x \in V$, set $v(x):=u (\Phi^{-1}(x))$. Using (\ref{phi_bound}), we deduce that $v \in W^{1,\infty}(V,\R^m) \cap C(V,\R^m)$. The set $V$ is a bounded open Lipschitz domain by construction. 
		
		In the following, we restrict ourselves to the case $m>2$. Let us fix an open rectangle $R$ containing $V$. Using Theorem \ref{cont density} and Theorem \ref{thm Gromov} we find $\eta=\eta(\varepsilon)>0$ and $\phi \in C^1(\overline{R},\R^m)$ such that 
		\begin{equation}\label{1005}
		\Vert v-\phi \Vert_{W^{1,p}(V)} \leq \frac{\varepsilon}{2}, \ \ \ \det((\nabla \phi(x))^T(\nabla \phi(x))) \geq \eta \ \ \mbox{for every $x \in R$}.
		\end{equation}
		Let $\sigma>0$ small which will be determined later. By Proposition \ref{affine inj} there exists $w \in \Aff^*(V,\R^m)$ such that
		\begin{align}\label{2}
			\Vert \phi-w \Vert_{W^{1,\infty}(V)} \leq \frac{\varepsilon}{2\mathcal{L}^n(V)}, \ \ \ \diam(\partial w(x)) \leq \sigma \ \ \mbox{for every $x \in V$}.
		\end{align}
		Choosing $\sigma$ suitably small \eqref{1005} and \eqref{2} also implies 
		\begin{equation}\label{1002}
		\det(M^TM) \geq \frac{\eta}{2} \ \ \mbox{for every $M \in \partial w(x)$ and every $x \in V$}.
		\end{equation}
		For every $x \in \Omega \setminus \overline{J_{u}}$, set $\tilde{u}(x):=w ( \Phi(x))$. We now check that the function $\tilde{u}$ has all the desired properties.
		By \eqref{phi_bound}, \eqref{1005} and (\ref{2}) we deduce that 
		$$
		\Vert u -\tilde{u} \Vert^p_{L^p(\Omega)} \leq 3 \Vert v -w \Vert^p_{L^p(V)} \leq 3 \varepsilon^p.
		$$
		By construction we have $J_{\tilde{u}} \subseteq J_{u}$. Moreover, since $\tilde{u}$ is piecewise affine and $J_u$ is polyhedral, we have that $J_{\tilde{u}}$ has a finite number of connected components, which implies $\tilde{u} \in \widehat{\mathcal{W}}(\Omega,\R^m)$.
		
		We want to prove $\tilde{u} \in \Aff^*(\Omega \setminus \overline{J_{\tilde{u}}}, \R^m)$. In particular, we show that
		\begin{equation}\label{property}
			\det(M^T M) \geq \frac{\eta}{8} \ \ \mbox{for every $M \in \partial \tilde{u}(x)$ and every $x \in \Omega \setminus \overline{J_{\tilde{u}}}$}.
		\end{equation} 
		To this aim we denote by $\mathcal{T}$ the triangulation of $\tilde{u}$ on $\Omega \setminus \overline{J_{\tilde{u}}}$ and we fix $x \in \Omega \setminus \overline{J_{\tilde{u}}}$. If $x \in T \setminus \partial T$ for some $T \in \mathcal{T}$ then property \eqref{property} follows immediately from (\ref{phi_bound}) and \eqref{1002}. Otherwise let $K \in \N$ and $T_1,\dots,T_K$ be all the elements of $\mathcal{T}$ containing $x$. Let $M \in \partial \tilde{u}(x)$. By construction $$\partial \tilde{u}(x)=\textnormal{conv}\left(\left\{ \nabla w(\Phi(x_k)) D\Phi(x_k) \right\}_{k=1}^K\right),$$ 
		where $x_k \in T_k \setminus \partial T_k$ for every $k=1,\dots,K$. Hence there exist $\lambda_k \geq 0$ such that $\sum_k \lambda_k=1$ and $M=\sum_k \lambda_k\nabla w(\Phi(x_k)) D\Phi(x_k)$. We have $|M| \leq 3\Vert \nabla \phi \Vert_{L^\infty}$. Let $\omega$ be the uniform modulus of continuity of the function $\det: \ \mathbb{M}^{2 \times 2} \to \R$ restricted to the closure of the ball $B_{9 \Vert \nabla \phi \Vert_{L^\infty}^2}(0)$. Take $\sigma$ small such that $\omega(6\sigma \Vert \nabla \phi \Vert_{L^\infty}+18\sigma) \leq \eta/8$. We also notice that $\nabla w(\Phi(x_k)) \in \partial w(\Phi(x))$ and $D\Phi(x_k) \in \partial \Phi(x)$ for every $k=1,\dots,K$. Thus, using \eqref{phi_bound}, \eqref{2}, \eqref{1002}, we have
		\begin{align}
			\label{2001}
			\det(M^TM) & \geq  \det \left( \left(\sum_{k=1}^K \lambda_k\nabla w(\Phi(x_1)) D\Phi(x_k)\right)^T \left(\sum_{k=1}^K \lambda_k\nabla w(\Phi(x_1)) D\Phi(x_k)\right)\right)+ \\
			\nonumber & \ \ \ \ \vphantom{\int} - \omega(6\sigma \Vert \nabla \phi \Vert_{L^\infty}+18\sigma) \\
			\nonumber & \geq \frac{\eta}{2} \det \left( \left(\sum_{k=1}^K \lambda_k D\Phi(x_k)\right)^T  \left(\sum_{k=1}^K \lambda_k D\Phi(x_k)\right)\right)- \omega(6\sigma \Vert \nabla \phi \Vert_{L^\infty}+18\sigma) \\
			\nonumber & \geq \frac{\eta}{4}- \omega(6\sigma \Vert \nabla \phi \Vert_{L^\infty}+18\sigma) \geq \frac{\eta}{8}.
		\end{align}
		This proves \eqref{property}.
		
		We are only left to estimate the $L^p$ distance between $\nabla \tilde{u}$ and $\nabla u$. Using (\ref{phi_bound}) \eqref{1005} and \eqref{2}, we have
		\begin{align*}
			\int_\Omega |\nabla \tilde{u}(x)-\nabla u(x)|^p \, dx & = \int_\Omega |\nabla (w (\Phi(x)))-\nabla(v( \Phi(x)))|^p \, dx \\  & = \int_\Omega |\nabla w( \Phi(x)) D\Phi(x)-\nabla v (\Phi(x)) D\Phi(x)|^p \, dx \\ & \leq 3^p \int_\Omega |\nabla w (\Phi(x))-\nabla v (\Phi(x))|^p \, dx \\ & \leq 3^p \int_V 3 |\nabla w(x')-\nabla v (x')|^p \, dx'  \leq 3^{p+1}\varepsilon^p.
		\end{align*}
		
		To conclude, the argument to prove the case $m=2$ is analogous to the one used for $m>2$. Actually it is simpler due to the definition of $\Aff^*$ when $m=n$.

	\end{proof}
	
	\begin{rem}
		\label{r:3.4}
		Proposition \ref{main result 1} can be immediately generalized to the setting $n > 2$ and $1<p \leq 2$ a sequence of functions $u_j \in \Aff^*(\Omega \setminus \overline{J_{u_j}},\R^m) \cap \mathcal{W}(\Omega,\R^m)$ satisfying properties (i)--(iii) for $m \geq n$ . Indeed in this case by Remark \ref{cap simplexes} we can assume that each approximant obtained using Theorem \ref{ct_approx} is such that the connected components of its jump set are ($n-1$)-dimensional simplexes compactly contained in $\Omega$. Hence, using Remark \ref{homom for n magg 2}, we can repeat the construction in the proof of Proposition \ref{main result 1}.
	\end{rem}

	\section{Dimension Reduction under Noncompenetration}
	
	\subsection{Statement of the Main Result}

	Let $1<p<\infty$, $\Sigma$ be a bounded open subset of $\R^2$ with Lipschitz boundary and set $\Sigma_\rho:=\Sigma \times (-\rho/2,\rho/2)$ for $\rho>0$, and we refer to \eqref{e:Grho}, \eqref{G0} and \eqref{reduced density} for the definitions of $\rho^{-1}\mathcal{G}_\rho$, $\mathcal{G}_0$ and the reduced density $W_0$. Our main result is the following.

	\begin{thm}\label{gamma limit teo}
		Let $\Sigma$ be an open bounded subset of $\R^2$ with Lipschitz boundary and $1<p<\infty$. Let $W:\mathbb{M}^{3 \times 3} \to [0,+\infty]$ fulfill \textnormal{($H_1$)}--\textnormal{($H_5$)}. Then 
		$$
		\Gamma-\lim_{\rho \to 0} \rho^{-1} \mathcal{G}_\rho = \mathcal{G}_0
		$$
		in $L^0(\Sigma_1,\R^3)$ with respect to the topology induced by the convergence in measure.
	\end{thm}
	
	The $\Gamma$-liminf inequality is analogous to the case without constraints on the determinant.
	
	\begin{prop}\label{gamma liminf}
		Let $u \in \mathcal{A}$. Then, for every sequence $u_\rho  \in \gsbvp(\Sigma_1,\R^3)$ such that $u_\rho \to u$ in measure in $\Sigma_1$, we have
		\begin{equation*}
			\mathcal{G}_0(u) \leq \liminf_{\rho \to 0} \rho^{-1}\mathcal{G}_\rho(u_\rho).
		\end{equation*}
	\end{prop}
	
	\begin{proof}
		If $\liminf_{\rho \to 0} \rho^{-1} \mathcal{G}_\rho(u_\rho)=+\infty$ there is nothing to prove. We can assume that $\liminf_{\rho \to 0} \rho^{-1} \mathcal{G}_\rho(u_\rho) < +\infty$. Up to a not relabeled subsequence we may assume that the liminf is a limit and that there exists $C>0$ such that $\sup_\rho\left\{\rho^{-1}\mathcal{G}_\rho(u_\rho)\right\} \leq C$. By ($H_4$) we can apply Ambrosio compactness Theorem \cite[Theorem 4.8]{Ambrosio2000FunctionsOB} and deduce the existence of $v \in \gsbvp(\Sigma_1,\R^3)$ such that, up to a further subsequence, $u_\rho \to v$ weakly in $\gsbvp(\Sigma_1,\R^3)$.
		Thus $u=v$ a.e. in $\Sigma_1$. Moreover, by ($H_4$), we have that for every $\rho>0$,
		\begin{align*}
			C \geq \rho^{-1} \mathcal{G}_\rho(u_\rho) \geq C_1 \int_{\Sigma_1} \left( |\nabla_\alpha u_\rho|^p+\frac{1}{\rho^p} |\partial_3 u_\rho|^p \right) \, dx -C_1.
		\end{align*}
		This, together with the weak convergence of $\nabla u_\rho$ to $\nabla u$ in $L^p(\Sigma_1,\mathbb{M}^{3 \times 3})$, implies that
		\begin{equation*}
			\int_{\Sigma_1} |\partial_3 u|^p \leq \liminf_{\rho \to 0} \int_{\Sigma_1} |\partial_3 u_\rho|^p=0.
		\end{equation*}
		Therefore, $\partial_3 u=0$ a.e. in $\Sigma_1$.
		
		For every $\tilde{\rho}>0$ it holds 
		\begin{align*}
			C & \geq \lim_{\rho \to 0} \rho^{-1} \mathcal{G}_\rho(u_\rho) \geq \liminf_{\rho \to 0} \int_{J_{u_\rho}} \psi_{\rho}(\nu_{u_\rho}) \, d \mathcal{H}^2 \geq \liminf_{\rho \to 0} \int_{J_{u_\rho}} \psi_{\tilde{\rho}}(\nu_{u_\rho}) \, d \mathcal{H}^2 \geq \int_{J_{u}} \psi_{\tilde{\rho}}(\nu_{u}) \, d \mathcal{H}^2 \\ & \geq \frac{1}{\tilde{\rho}} \int_{J_{u} } |(\nu_u)_3| \, d \mathcal{H}^2.
		\end{align*}
		From the arbitrariness of $\tilde{\rho}>0$ we infer that $(\nu_u)_3=0 \,$  $\mathcal{H}^2$-a.e. on $J_u$, hence $u$ does not depend on $x_3$ and $u \in \mathcal{A}$.
		
		Finally, by definition of $W_0$ in (\ref{reduced density}) we have
		\begin{align*}
			\liminf_{\rho \to 0} \rho^{-1}\mathcal{G}_\rho(u_\rho) & = \liminf_{\rho \to 0} \left( \int_{\Sigma_1} W(\nabla u_\rho) \, dx+ \mathcal{H}^2(J_{u_\rho}) \right) \\ & \geq \liminf_{\rho \to 0}  \int_{\Sigma_1} W_0(\nabla_\alpha u_\rho) \, dx+ \mathcal{H}^2(J_{u_\rho}) \\ & \geq \liminf_{\rho \to 0}  \int_{\Sigma_1} \mathcal{Q}W_0(\nabla_\alpha u_\rho) \, dx+ \mathcal{H}^2(J_{u_\rho}) \geq \mathcal{G}_0(u), 
		\end{align*}
		where in the last inequality we have used Theorem \ref{relaxation gsbvp}.
	\end{proof}

	We conclude this section by recalling some properties of $W_0$ (see \cite[Lemma 2.4]{hafsa2006nonlinear}) which will be used later on in the paper.
	
	\begin{lem}\label{reduced density lemma}
		Let $W: \mathbb{M}^{3 \times 3} \to [0,+\infty]$ satisfy \textnormal{($H_1$)}--\textnormal{($H_5$)}. Then, the following facts hold:
		\begin{enumerate}
			\item[\textnormal{(i)}] $W_0$ is continuous as an extended valued function and satisfies \textnormal{($H_4$)} with the same $C_1$ and $p$ of $W$;
			\item[\textnormal{(ii)}] for $A \in \mathbb{M}^{3 \times 2}$, $W_0(A)= +\infty$ if and only if $A^1 \wedge A^2 =0$,
			\item[\textnormal{(iii)}] for every $\delta >0$, there exists $c_\delta>0$ such that for every $A \in \mathbb{M}^{3 \times 2}$, if $|A^1 \wedge A^2| \geq \delta$, then $W_0(A) \leq c_\delta(1+|A|^p)$.
		\end{enumerate}
	\end{lem}
	
	\subsection{An intermediate step: $\Gamma$-limsup inequality in terms of an auxiliary functional}
	We start by estimating the $\Gamma$-limsup of the functionals $\rho^{-1}\mathcal{G}_\rho$ in terms of an auxiliary functional defined on a subset of the functions in $\gsbvp(\Sigma_1,\R^3)$ not dependent on the third variable. To this aim, we start by fixing some notations. 
	
	Let $\Sigma \subset \R^2$ be an open set and let $u$ be a function defined on $\Sigma_1=\Sigma \times (-1/2,1/2)$ not depending on the third variable. In the following, to short the notation, we will denote with $u$ also the corresponding trace of $u$ on $\Sigma$. Set $Y \subset \gsbvp(\Sigma_1,\R^3)$ as
	\begin{align*}
		Y:= \Big\{ u \in \gsbvp(\Sigma_1,\R^3): & \ u(x_1,x_2,x_3)=u|_\Sigma(x_1,x_2), \, u \in \Aff^*(\Sigma \setminus \overline{J_u},\R^3) \cap \widehat{\mathcal{W}}(\Sigma,\R^3) \Big\},
	\end{align*}
	observe also that is immediate to see that $\mathcal{H}^2(J_u)=\mathcal{H}^1(J_u \cap \Sigma)$ for every $u \in Y$. 
	
	We introduce the functional
	\begin{equation}
		\label{e:G0-def}
		\mathcal{G}_0^w(u) =
		\begin{cases*}
			\displaystyle \int_{\Sigma_1} W_0(\nabla u) \, dx +\mathcal{H}^2(J_u) & if $u \in Y$, \\
			+\infty & otherwise.
		\end{cases*}
	\end{equation}
	
	\begin{thm}\label{gamma limsup prop}
		Let $u \in  \gsbvp(\Sigma_1,\R^3)$. Then 
		\begin{equation}\label{G_limsup relaxed}
			\Gamma-\limsup_{\rho \to 0} \rho^{-1} \mathcal{G}_\rho (u) \leq \overline{\mathcal{G}_0^w}(u)
		\end{equation} 
		where $\overline{\mathcal{G}_0^w}$ is the relaxation of $\mathcal{G}_0^w$.
	\end{thm}
	
	The proof of Theorem \ref{gamma limsup prop} needs some preliminary work.

	We now state a result that provides a suitable smooth approximation for functions in $\Aff^*$ in the spirit of \cite[Lemma 7]{benbe1996}.

	\begin{lem}\label{trabelsi lemma modified}
		Let $U$ be an open bounded subset of $\R^2$ with Lipschitz boundary and let $v \in \textnormal{Aff}^*(U,\R^3)$.
		Then, there exists a sequence $v_j \in C^1(\overline{U},\R^3)$ with the following properties:
		\begin{enumerate}
			\item[\textnormal{(i)}] $v_j \to v$ in $W^{1,p}(U,\R^3)$ and $\Vert \nabla v_j \Vert_{L^\infty} \leq 2 \Vert \nabla v \Vert_{L^\infty}$;
			\item[\textnormal{(ii)}] there exists $\theta >0$ such that for every $j \geq 1$
			\begin{equation}\label{theta+}
				|\partial_1 v_j(x) \wedge \partial_2 v_j(x)| \geq \theta \ \ \ \mbox{for every $x \in \overline{U}$}.
			\end{equation}
		\end{enumerate}
	\end{lem}
	
	\begin{proof}

		For $\varepsilon>0$ let us denote	$$(U)_{-\varepsilon}:=\left\{ x \in U: \ \dist(x,\partial U) > \varepsilon \right\}.$$
		Let $\varepsilon_j \to 0$ as $j \to +\infty$ such that for every $x \in (U)_{-2\varepsilon_j}$ there exists $z_x \in B_{\varepsilon_j}(x)$ satisfying $\nabla v(y) \in \partial v(z_x)$ for every $y \in B_{\varepsilon_j}(x)$.

		
		Let $\left\{ \rho_\tau \right\}_{\tau>0}$ be a family of standard mollifiers.
		We define $u_j:=\rho_{\varepsilon_j} * v \in C^1((U)_{-2\varepsilon_j},\R^3)$.
		We clearly have $\Vert \nabla u_j \Vert_{L^\infty} \leq \Vert \nabla v \Vert_{L^\infty}$. 
		Since $v \in \Aff^*(U,\R^3)$, we can fix $\eta>0$ such that $|A^1 \wedge A^2| \geq \eta$ for every $A \in \partial v(x)$, for every $x \in U$.
		By the properties of convolution, we have that for every $x \in (U)_{-2\varepsilon_j}$, $\nabla u_j(x) \in \partial u(z_x)$. This in turn implies that $|\partial_1 u_j(x) \wedge \partial_2 u_j(x)| \geq \eta$ for every $x \in  (U)_{-2\varepsilon_j}$. By uniform continuity of $\nabla u_j$, we can extend this property to $ \overline{(U)_{-2\varepsilon_j}}$. 

		Since $U$ has Lipschitz boundary, there exists a sequence of diffeomorphisms $\phi_j \in C^1(U,U)$ such that $\phi_j(U) \subseteq (U)_{-2\varepsilon_j}$ and $\Vert \phi_j - \textnormal{Id} \Vert_{W^{1,\infty}(U)} \to 0$ as $j \to +\infty$. Setting $v_j(x):=u_j(\phi_j(x)) $, we have $v_j \in C^1(U,\R^3)$, $\Vert v_j \Vert_{L^\infty} \leq \Vert v \Vert_{L^\infty}$ and $\Vert \nabla v_j \Vert_{L^\infty}  \leq 2\Vert \nabla v \Vert_{L^\infty}$. Furthermore, for $j$ large enough it holds
		\begin{align*}
			|\partial_1 v_j(x) \wedge \partial_2 v_j(x)|=\det(D\phi_j(x))|\partial_1 u_j(\phi_j(x)) \wedge \partial_2 u_j(\phi_j(x))| \geq \frac{\eta}{2} \ \ \mbox{for every $x \in \overline{U}$}.
		\end{align*}
		Therefore, \eqref{theta+} is satisfied with $\theta:=\eta/2$.
		
	\end{proof}

	In the following three lemmas we give some results about the representation of the integral of $W_0$.
	
	\noindent Let us start giving a general result about $W$ and its reduced counterpart $W_0$.
	
	\begin{lem}\label{lem interchange}
		Let $\alpha,K>0$ and set
		$$
		\Lambda (\alpha,K):= \left\{ A \in \mathbb{M}^{3 \times 2} : \ |A^1 \wedge A^2| \geq \alpha, \ |A| \leq K \right\}.
		$$
		Then, there exists $\beta=\beta(\alpha,K)>0$ such that for every $A \in \Lambda(\alpha,K)$ we have 
		\begin{equation}\label{99}
			W_0(A)=\min \left\{ W(A|\zeta): \ \zeta \in \R^3 , \ \det(A|\zeta) \geq \frac{1}{\beta}, \  \ |\zeta| \leq \beta \right\}.
		\end{equation}
		
	\end{lem}
	
	\begin{proof}
		By definition of $W_0$, to prove \eqref{99} we show that there exists $\beta=\beta(\alpha,K)>0$ such that for every $A \in \Lambda(\alpha,K)$, the following holds: 
		\begin{equation}\label{199}
			\mbox{for every $\zeta \in \R^3: \ W(A|\zeta) \leq W_0(A)+1$} \ \Longrightarrow \ \det(A|\zeta) \geq \frac{1}{\beta} \ \mbox{and} \ |\zeta| \leq \beta.
		\end{equation}
		
		Let $A \in \Lambda(\alpha,K)$ and $\zeta \in \R^3$ be such that $W(A|\zeta) \leq W_0(A)+1$. By the assumptions ($H_4$) and ($H_5$) on $W$, we have that
		\begin{align*}
			C_1(|\zeta|^p-1)  \leq W(A|\zeta) \leq W_0(A)+1 \leq W\left( A \, \bigg| \, \frac{A^1 \wedge A^2}{|A^1 \wedge A^2|} \right)+1   \leq c_\alpha (|A|^p+2)+1.
		\end{align*}
		This implies that $|\zeta| \leq \beta_1:=(c_\alpha K^p+2c_\alpha+C_1+1)^{1/p}$.
		
		As for the lower bound on the determinant, let us assume by contradiction that for every $n \in \N$ there exist $A_n \in \Lambda(\alpha,K)$ and $\zeta_n \in \R^3$ such that $W(A_n|\zeta_n) \leq W_0(A_n)+1$ and $\det(A_n|\zeta_n) \leq 1/n$. By the previous step we have that $|\zeta_n| \leq \beta_1$ for every $n \geq 1$. By compactness of $\Lambda(\alpha,K)$, there exist $A \in \Lambda(\alpha,K)$ and $\zeta \in \R^3$ such that, up to a not relabeled subsequence, $A_n \to A$ and $\zeta_n \to \zeta$. This implies that 
		\begin{equation}\label{200}
			0=\lim_{n \to +\infty} \det(A_n|\zeta_n)=\det(A|\zeta).
		\end{equation}
		By continuity of $W_0$ (see Lemma \ref{reduced density lemma}) we have $W_0(A_n) \to W_0(A)$. Since $A \in \Lambda(\alpha,K)$, we have that $W_0(A) < +\infty$. On the other hand, in view of \eqref{200} and of hypotheses ($H_1$) and ($H_3$) on $W$ we get 
		$$
		W_0(A)+1 = \lim_{n \to +\infty} W_0(A_n) +1 \geq \lim_{n \to +\infty} W(A_n|\zeta_n) = +\infty,
		$$
		which is a contradiction. Hence, there exists $\beta_2>0$ such that for every $A \in \Lambda(\alpha,K)$ and $\zeta \in \R^3$ with $W(A|\zeta) \leq W_0(A)+1$, it holds $\det(g(x)|\zeta) \geq 1/\beta_2$. Setting $\beta:=\max \left\{ \beta_1,\beta_2 \right\}$ we infer \eqref{199} and hence \eqref{99}.
	\end{proof}
	
	
	\begin{lem}\label{lem anza-hafsa}
		Let $\Omega \subset \R^2$ be an open bounded set, $\eta>0$, and $G: \overline{\Omega} \to \mathbb{M}^{3 \times 2}$ be an uniformly continuous function such that $|G^1(x) \wedge G^2(x)| \geq \eta$ for every $x \in \overline{\Omega}$. Let $\Lambda_j: \overline{\Omega} \rightrightarrows \R^3$ be the multifunction defined for $j \geq 1$ by
		$$
		\Lambda_j(x):= \left\{ \zeta \in \R^3 : \ \det(G(x)|\zeta) \geq \frac{1}{j} \ \ \mbox{and} \ \ |\zeta| \leq j \right\} \ \ \mbox{for $x \in \overline{\Omega}$}.
		$$ 
		Then, there exists $j(\eta,\Vert G \Vert_{L^\infty})$ such that for every $j \geq j(\eta,\Vert G \Vert_{L^\infty})$ we have
		\begin{equation}\label{120}
			\inf_{\phi \in C^\infty(\overline{\Omega},\Lambda_j)} \int_\Omega W(G(x)|\phi(x)) \, dx = \int_\Omega \ \inf_{\zeta \in \Lambda_j(x) } W(G(x)|\zeta) \, dx.
		\end{equation} 
	\end{lem}
	
	\begin{proof}
		
		By Tietze extension theorem and uniform continuity of $G$, we may assume that $G$ satisfies the assumptions of the Lemma in a neighborhood $\Omega'$ of $\Omega$, up to taking a smaller~$\eta$. 
		
		Let $h>0$ and let $\mathcal{T}_h$ be a regular triangulation of $\R^2$ such that $\diam(T) \leq h$ for every $T \in \mathcal{T}_h$. We fix $G_h$ any piecewise constant interpolation of $G$ on $\mathcal{T}_h$. Using Lemma \ref{lem interchange} with $j \geq \beta(\eta,\Vert G \Vert_{L^\infty})$ from \eqref{99}, for every $T \in \mathcal{T}_h$ we pick $\varphi_T \in \R^3$ such that $W(G_h(x)|\varphi_T) = W_0(G_h(x))$ for every $x \in T$. 
		We define $\phi_h\colon \Omega' \to \R^3$ as $\phi_h(x)=\varphi_T$ for $x \in T$.
		For $\delta>0$, we notice that by uniform continuity of $G$ and by continuity of $W$, for every $h$ small enough it holds 
		\begin{equation}\label{aaa}
			W(G(x)|\phi_h(x)) \leq W_0(G(x))+\delta \ \ \ \mbox{for every $x \in \Omega$}.
		\end{equation}
		Observe that by uniform continuity of $G$ we have 
		$$
		\det(G(x)|\phi_h(x)) \geq \frac{1}{2j} \ \ \ \mbox{for every $x \in \Omega'$}
		$$
		hence, we can find an uniform $\varepsilon>0$ not depending on $x_0$ and $h$, such that
		$$
		\det(G(x_0)|\phi_h(x)) \geq \frac{1}{3j} \ \ \ \mbox{for every $x_0 \in \Omega'$ and every $x \in B_\varepsilon(x_0) \cap \Omega'$}.
		$$
		Let $\{ \rho_\epsilon \}_\epsilon$ be a family of standard mollifiers and set $\phi_{\varepsilon,h}=\rho_\varepsilon * \phi_h$. By linearity of the determinant as a function of the columns and by the properties of convolutions we deduce that $\det(G(x)|\phi_{\varepsilon,h}(x)) \geq \frac{1}{3j}$ for every $x \in \Omega$.
		
		For every $h>0$ we have that $\phi_{\varepsilon,h} \to \phi_h$ in $L^p(\Omega,\R^3)$ for every $p<\infty$. By the uniform control of the determinant, by ($H_3$) and by \eqref{aaa} we have
		\begin{equation*}
			\lim_{\varepsilon \to 0} \int_\Omega W(G(x)|\phi_{\varepsilon,h}) \, dx = \int_\Omega W(G(x)|\phi_h(x)) \, dx \leq \int_\Omega W_0(G(x)) \, dx+\delta|\Omega|.
		\end{equation*}
		Since $\delta$ is arbitrary we conclude.

	\end{proof}

	\begin{cor}\label{interchange cor}
		Let $\Omega \subset \R^2$ be an open bounded set, $\eta>0$, and $G: \overline{\Omega} \to \mathbb{M}^{3 \times 2}$ be uniformly continuous and such that $|G^1(x) \wedge G^2(x)| \geq \eta$ for every $x \in \overline{\Omega}$. Then, there exists $\beta=\beta(\eta,\Vert G \Vert_{L^\infty})$ such that
		\begin{align*}
			\int_\Omega W_0(G(x)) \, dx = \inf \left\{ \int_\Omega W \right. & (G(x)|\phi(x)):  \phi \in C^\infty(\overline{\Omega},\R^3), \\ &  \det(G(x)|\phi(x)) \geq \frac{1}{\beta}  \  \mbox{for}  \ \mbox{$x \in \overline{\Omega}$,} \ \Vert \phi \Vert_{L^\infty} \leq \beta \bigg\}.
		\end{align*}
	\end{cor}
	
	\begin{proof}
		It is enough to combine Lemma \ref{lem anza-hafsa} with Lemma \ref{lem interchange} with the choice $\alpha=\eta$ and $K=\Vert G \Vert_{L^\infty}$.
	\end{proof}
	
	
	We now prove an analogous result to Corollary \ref{interchange cor}, which gives that for any subset $U \subset \Omega$ there exists an almost optimal function $\phi$ for $W_0$ on $U$ which satisfies the determinant constraint in the whole $\Omega$. 
	
	\begin{lem}\label{lem interchange subset}
		Let $\Omega$ and $G$ be as in Corollary \ref{interchange cor}. Then, for every open set $U \subset \Omega$, there exists $\beta=\beta(\eta,\Vert G \Vert_{L^\infty})>0$ such that the following holds
		\begin{align*}
			\int_U W_0(G(x)) \, dx = \inf \left\{ \int_U W \right. & (G(x)|\phi(x)):  \phi \in C^\infty(\overline{\Omega},\R^3), \\ &  \det(G(x)|\phi(x)) \geq \frac{1}{\beta}  \  \mbox{for}  \ \mbox{$x \in \overline{\Omega}$,} \ \Vert \phi \Vert_{L^\infty(\Omega)} \leq \beta \bigg\}.
		\end{align*}
	\end{lem}

	\begin{proof}
		Fix $U \subset \Omega$ open and $\varepsilon>0$. Using Corollary \ref{interchange cor} we can find $\phi \in C^\infty(\overline{\Omega},\R^3)$ such that $\Vert \phi \Vert_{L^\infty(\Omega)} \leq \beta$, $\det(G(x)|\phi(x)) \geq 1/\beta$ and
		\begin{align}
			\label{701}
			\int_\Omega W(G(x)|\phi(x)) \, dx \leq \int_\Omega W_0(G(x)) \, dx+\varepsilon.
		\end{align}
		Noticing that $W(G(x)|\phi(x)) \geq W_0(G(x))$ for every $x \in \Omega$, \eqref{701} yields
		\begin{equation}\label{702}
			\int_U W(G(x)|\phi(x)) \, dx \leq \int_U W_0(G(x)) \, dx+\varepsilon;
		\end{equation}
		and therefore the proof of the Lemma.
	\end{proof}
	
	\begin{lem}\label{selection}
		
		Let $E \subset \R^2$ be open and bounded, let $\eta>0$ and $G \in C(\overline{E},\mathbb{M}^{3 \times 2})$ be such that $|G^1(x) \wedge G^2(x)| \geq \eta$ for every $x \in \overline{E}$. Let $\varepsilon\in(0,1/2)$, $\mathcal{T}$ be a finite triangulation of $E$ and $\Psi: E \to \mathbb{M}^{2 \times 2}$ be a piecewise constant function on the triangulation such that 
		\begin{equation}\label{300}
			\Vert \det(\Psi(x))-1 \Vert_{L^\infty} \leq \varepsilon.
		\end{equation}
		Then, there exists $\varphi \in C^\infty(\overline{E},\R^3)$ and $\beta=\beta(\eta,\Vert G \Vert_{L^\infty},\Vert \Psi \Vert_{L^\infty})>0$ such that $\Vert \varphi \Vert_{L^\infty(E)} \leq \beta$ and
		\begin{align}
			\label{399}
			& \det(G(x) \Psi(x)|\varphi(x)) \geq \frac{1}{\beta}, \\
			\label{400}
			& \int_E W(G(x)\Psi(x)|\varphi(x)) \, dx \leq \int_E W_0(G(x)\Psi(x)) \, dx+\varepsilon.
		\end{align}
		
	\end{lem}
	
	\begin{proof}
		
		Let $\mathcal{T}=\left\{ V_i \right\}_{i=1}^N$ and $\Psi|_{V_i}=A_i$ where $A_i \in \mathbb{M}^{2 \times 2}$ and $|\det(A_i)-1| \leq \varepsilon$ for every $i=1, \dots, N$. We apply Lemma \ref{lem interchange subset} to the function $G A_i$ on the subset $V_i$ of $E$ for every $i=1,\dots,N$. We find $\tilde{\beta}=\tilde{\beta}(\eta,\Vert G \Vert_{L^\infty},\Vert \Psi \Vert_{L^\infty})$ and, for every $i=1,\dots,N$, $\varphi_i \in C^\infty(\overline{E},\R^3)$ such that $\Vert \varphi_i \Vert_{L^\infty(E)} \leq \tilde{\beta}$, $\det(G(x)A_i|\varphi_i(x)) \geq 1/\tilde{\beta}$ for every $x \in \overline{E}$ and 
		$$
		\int_{V_i} W(G(x)A_i|\varphi_i(x)) \, dx \leq \int_{V_i} W_0(G(x)A_i) \, dx + \frac{\varepsilon}{2N}.
		$$
		Let $U_i$ be a small open neighborhood of $V_i$ for every $i=1,\dots,N$. Take $\left\{ \mu_i \right\}_i$ a partition of unity subordinate to the open cover $\left\{U_i\right\}_{i=1}^N$ of $E$. We define $\varphi:=\sum_{i=1}^N \mu_i \varphi_i$. Observe that $\Vert \varphi \Vert_{L^\infty(E)} \leq \tilde{\beta}$ by construction. Furthermore, since $\det(G(x)|\varphi_i(x)) \geq 1/((1+\varepsilon)\tilde{\beta})$ for every $i=1,\dots,N$, for every $x \in \overline{V_k}$ with $k=1,\dots,N$, we have
		\begin{align}\label{det}
			\begin{split}
				\det(G(x)\Psi(x)|\varphi(x)) & =\sum_{i=1}^N \mu_i(x) \det(G(x)A_k|\varphi_i(x))\\ & = \sum_{i=1}^N \mu_i(x) \det(G(x)|\varphi_i(x))\det(A_k) \geq \frac{1-\varepsilon}{(1+\varepsilon)\tilde{\beta}} \geq \frac{1}{3 \tilde{\beta}}.
			\end{split}
		\end{align}
		Setting $\beta:= 3\tilde{\beta}$, we obtain $\Vert \varphi \Vert_{L^\infty(E)} \leq \beta$ and \eqref{399}.
		Finally, using ($H_5$) together with \eqref{det},
		\begin{align*}
			\int_E \! W(G(x)\Psi(x)|\varphi(x)) dx & \leq \sum_{i=1}^N \int_{V_i} \! W(G(x)\Psi(x)|\varphi_i(x))  dx+\sum_{i=1}^N \int_{(U_i \setminus V_i)\cap E} \! W(G(x)\Psi(x)|\varphi(x)) dx \\
			& \leq \int_E W_0(G(x)\Psi(x)) dx + \frac{\varepsilon}{2}+c_{\beta}(\Vert \Psi \Vert_{L^\infty}^p\Vert G \Vert_{L^\infty}^p+\beta^p) \sum_{i=1}^N |U_i \setminus V_i|.
		\end{align*}
		Hence, by taking $|U_i \setminus V_i|$ suitable small for every $i=1,\cdots,N$, we prove \eqref{400}.
		
	\end{proof}

	We are now in a position to give the proof of Theorem \ref{gamma limsup prop}. 
	
	\begin{proof}[Proof of Theorem \ref{gamma limsup prop}]
		Since $u \in Y$ there exists $\eta>0$ such that $|A^1 \wedge A^2| \geq \eta$ for every $A \in \partial u(x)$ and every $x \in \Sigma$. 
		We claim that for every $\varepsilon>0$ small there exist a sequence $\left\{ u_{j,\rho} \right\}_{j,\rho} \subset \sbvp(\Sigma_1,\R^3)$ such that $u_{j,\rho} \to u$ in $L^p(\Sigma_1 ,\R^3)$ as $\rho \to 0$ and $j \to \infty$, and
		\begin{align}\label{G_limsup}
			\begin{split}
				\limsup_{\ j \to \infty, \ \rho \to 0} \ & \int_{\Sigma_1} W \left(\nabla_\alpha u_{j,\rho} \bigg|  \frac{1}{\rho} \partial_3 u_{j,\rho}  \right) \, dx + \int_{J_{u_{j,\rho}}} \psi_\rho(\nu_{u_{j,\rho}}) \, d \mathcal{H}^2 \leq (1+\varepsilon)^2\mathcal{G}^w_0(u)+2\varepsilon,
			\end{split}
		\end{align}
		where $\nabla_\alpha u=(\partial_1 u|\partial_2 u)^T$ and the limsup is taken in the given order, that is, first $\rho \to 0$ and then $j \to +\infty$.
		We write $J_u = \cup_{\ell=1}^m \gamma_\ell$ where $\gamma_\ell$ are the connected components of $J_u$.
		
		Fix $\varepsilon>0$. Let $\delta_0>0$ be such that $(\gamma_\ell)_{\delta_0} \subset \Sigma$ and $(\gamma_\ell)_{\delta_0} \cap (\gamma_s)_{\delta_0}$ for every $s \neq \ell$. For every $0<\delta<\delta_0$, let $\Phi_\delta: \Sigma \setminus J_u \to \Omega^\delta $ by composition of the diffeomorphisms given by Lemma \ref{aff diffeo lemma} in correspondence of each connected component $\gamma_\ell$. Where the set $\Omega^\delta$ is a Lipschitz set such that $\Omega^\delta \subset \Sigma \setminus J_u$. In particular we may assume that for every $\delta \in (0,\delta_0)$
		\begin{align}\label{det eps}
			& \Vert D \Phi_\delta \Vert_{L^\infty} + \Vert D\Phi_\delta^{-1} \Vert_{L^\infty} \leq 3, \ \ \ \Vert \det (D\Phi_{\delta})-1 \Vert_{L^\infty} \leq \varepsilon; \\
			\label{2000}
			& |\det(M)-1| \leq \varepsilon \ \ \mbox{for every $M \in \partial \Phi_\delta^{-1}(x)$ and every $x \in \Omega^\delta$}.  
		\end{align}
		
		Let us set $\Omega^\delta_1=\Omega^\delta \times (-1/2,1/2)$. With a slight abuse of notation, for every $(x_1,x_2) \in \Omega^\delta$, we also set
		\begin{equation*}
			\Phi_\delta(x_1,x_2,x_3):=(\Phi_\delta(x_1,x_2),x_3).
		\end{equation*} 
		Now let us fix $\delta \in (0,\delta_0)$ and define the function $v:=u \circ \Phi_\delta^{-1}$ on $\Omega^\delta_1$. Using \eqref{det eps} and (\ref{2000}), and arguing as in \eqref{2001} in the proof of Theorem \ref{main result 1}, we have that $v \in \Aff^*(\Omega^\delta,\R^3)$ with
		$\| \nabla v \|_{L^\infty} \leq 3 \| \nabla u \|_{L^\infty}$.

		
		We apply Lemma \ref{trabelsi lemma modified} to the function $v$ on $\Omega^\delta$. Hence, we find  
		$
		v_j \in C^1(\overline{\Omega^{\delta}},\R^3) 
		$
		satisfying properties (i) and (ii) with some constant $\theta>0$ not depending on $j$.

		We now apply Lemma \ref{selection} with $G=\nabla v_j$, $\Psi=D\Phi_\delta(\Phi_\delta^{-1})$ and $E=\Omega^\delta$ for every $j \geq 1$. We find ${\varphi}_{j} \in C^\infty(\Omega^\delta,\R^3) \cap W^{1,\infty}(\Omega^\delta,\R^3)$ and $\beta=\beta(\theta, \Vert \nabla u \Vert_{L^\infty(\Sigma_1)})$ such that for every $j \geq 1$, $\Vert \varphi_{j} \Vert_{L^\infty(\Omega_1^\delta)} \leq \beta$ and
		\begin{align} 
			\label{401}
			& \det(\nabla_\alpha v_j (x) D\Phi_\delta(\Phi_\delta^{-1}(x))|\varphi_{j}(x)) \geq \frac{1}{\beta} \ \ \ \mbox{for a.e $x \in \Omega_1^\delta$}, \\
			\label{sup def}
			& \int_{\Omega^\delta_1} W(\nabla_\alpha v_j (x) D\Phi_\delta(\Phi_\delta^{-1}(x))|\varphi_{j}(x)) \, dx \leq \int_{\Omega^\delta_1} W_0(\nabla_\alpha v_j(x)D\Phi_\delta(\Phi_\delta^{-1}(x))) \, dx + \varepsilon.
		\end{align}

		We define the recovery sequence as
		\begin{equation*}
			u_{j,\rho}(x):=v_j( \Phi_\delta(x)) + \rho x_3 \varphi_{j} ( \Phi_\delta(x)).
		\end{equation*}
		By construction $u_{j,\rho} \in \sbvp(\Sigma_1,\R^3)$. Observe that the approximate normal $\nu_{u_{j,\rho}}$ to $J_{u_{j,\rho}}$ has third component $(\nu_{u_{j,\rho}})_3=0$, so that $\psi_\rho(\nu_{u_{j,\rho}})=1$ on $J_{u_{j,\rho}}$. Moreover, by definition of $\Phi_\delta$ it holds $J_{u_{j,\rho}} \subseteq J_u$. Thus, we have
		\begin{align}\label{conv of jump set}
			\begin{split}
				\limsup_{j \to +\infty} \, \limsup_{\rho \to 0} & \int_{J_{u_{j,\rho}}} \psi_\rho(\nu_{u_{j,\rho}}) \, d \mathcal{H}^2  = \limsup_{j \to +\infty} \, \limsup_{\rho \to 0} \mathcal{H}^2 (J_{u_{j,\rho}}) \leq \mathcal{H}^2(J_u).
			\end{split}
		\end{align}
		By the change of variables induced by $\Phi_\delta^{-1}$, we obtain
		\begin{align*}
			\limsup_{j \to \infty} \, \limsup_{\rho \to 0} \| u_{j,\rho}-u \|_{L^p(\Sigma_1)} & =\limsup_{j \to \infty} \|v_j \circ \Phi_\delta- u \|_{L^p(\Sigma_1)} \\ & \leq \lim_{j \to \infty} 3^{1/p} \|v_j-v \|_{L^p(\Omega_1^\delta)} =0.
		\end{align*}
		
		We now estimate $\rho^{-1}\mathcal{G}_\rho(u_{j,\rho})$. Since we have already computed the limit of the surface integral in (\ref{conv of jump set}), we restrict ourselves to the bulk part.
		Let us keep in mind that $\nabla \varphi_{j} \in L^\infty(\Omega_1^\delta ,\R^3)$ for $j$ fixed, although possibly not uniformly. However this is no problem since we first take the limit in $\rho$. By \eqref{399}, we deduce
		\begin{align*}
			&\limsup_{ j \to \infty} \, \limsup_{\rho \to 0} \int_{\Sigma_1} W \left(\nabla_\alpha u_{j,\rho} \left|  \frac{1}{\rho} \partial_3 u_{j,\rho} \right. \right) \, dx \\
			& = \limsup_{j \to \infty} \, \limsup_{\rho \to 0}  \int_{\Sigma_1} W\bigg(\nabla_\alpha v_j (\Phi_\delta(x))D\Phi_\delta(x)+\rho x_3 \nabla_\alpha \varphi_{j} (\Phi_\delta(x)) D\Phi_\delta(x) \bigg| \varphi_{j} (\Phi_\delta(x)) \bigg) \, dx \\
			& = \limsup_{j \to \infty}  \int_{\Sigma_1} W\bigg(\nabla_\alpha v_j (\Phi_\delta(x))D\Phi_\delta(x) \bigg| \varphi_{j} (\Phi_\delta(x)) \bigg) \, dx.
		\end{align*}
		We now perform a change of variables with $\Phi_\delta^{-1}$. After that, using (\ref{det eps}) and (\ref{sup def}), we get
		\begin{align*}
			& \limsup_{j \to \infty}  \int_{\Sigma_1} W\bigg(\nabla_\alpha v_j (\Phi_\delta(x))D\Phi_\delta(x) \bigg| \varphi_{j} (\Phi_\delta(x)) \bigg) \, dx \\
			& =  \limsup_{j \to \infty} \int_{\Omega_1^\delta} W \bigg( \nabla_\alpha v_j(x')D\Phi_\delta(\Phi_\delta^{-1}(x')) \bigg| \varphi_{j}(x') \bigg) \, |\det(D\Phi^{-1}_\delta(x'))| dx'\\
			& \leq \limsup_{j \to \infty} \ (1+\varepsilon) \int_{\Omega_1^\delta} W \bigg(\nabla_\alpha v_j(x')D\Phi_\delta(\Phi_\delta^{-1}(x')) \bigg| \varphi_{j}(x')\bigg) \, dx' \\
			& \leq \limsup_{j \to \infty} \ (1+\varepsilon) \int_{\Omega_1^\delta} W_0 \bigg( \nabla_\alpha v_j(x')D\Phi_\delta(\Phi_\delta^{-1}(x')) \bigg) \, dx'+ 2 \varepsilon.
		\end{align*}
		We can now pass to the limit on $j$ by Vitali dominate convergence theorem using properties (i)--(iv) of Lemma \ref{trabelsi lemma modified} and Lemma \ref{reduced density lemma}. Finally, by a change of variables with $\Phi_\delta$, in view of \eqref{det eps} we have 
		\begin{align*}
			&\limsup_{j \to \infty} \ (1+\varepsilon) \int_{\Omega_1^\delta} W_0 \bigg( \nabla_\alpha v_j(x')D\Phi_\delta(\Phi_\delta^{-1}(x')) \bigg) \, dx'+ 2 \varepsilon \\
			& =  (1+\varepsilon) \int_{\Omega_1^\delta} W_0 \bigg(\nabla_\alpha v(x')D\Phi_\delta(\Phi_\delta^{-1}(x')) \bigg) \, dx'+ 2 \varepsilon \\
			& \leq (1+\varepsilon)^2 \int_{\Sigma_1} W_0 \bigg(\nabla_\alpha v(\Phi_\delta(x))D\Phi_\delta(x) \bigg) \, dx+ 2 \varepsilon \\
			& = (1+\varepsilon)^2 \int_{\Sigma_1} W_0(\nabla_\alpha u(x)) \, dx+ 2 \varepsilon.
		\end{align*}
		Thus, keeping in mind \eqref{conv of jump set}, (\ref{G_limsup}) is proved.
		
		By the arbitrariness of $\varepsilon$, we infer $\Gamma-\limsup_{\rho \to 0} \rho^{-1}\mathcal{G}_\rho(u) \leq \mathcal{G}_0^w(u)$ for every $u \in Y$.
		For the remaining $u \in \gsbvp(\Sigma_1,\R^3) \setminus Y$, the above inequality is obvious. By lower semicontinuity of the $\Gamma$-limsup (see \cite[Proposition 6.8]{DM93gamma}) we then conclude.
	\end{proof}

	%
	
	\subsection{Proof of the Main Theorem} 
	
	In this last subsection we give the proof of Theorem \ref{gamma limit teo}. Let us summarize our strategy. So far, using Proposition \ref{gamma liminf} and Theorem \ref{gamma limsup prop}, for every $u \in \gsbvp(\Sigma,\R^3)$ we have proven that 
	
	\begin{equation*}
		\mathcal{G}_0(u) \leq \Gamma-\liminf_{\rho \to 0} \rho^{-1} \mathcal{G}(u) \leq \Gamma-\limsup_{\rho \to 0} \rho^{-1} \mathcal{G}(u) \leq \overline{\mathcal{G}_0^w}(u).
	\end{equation*}
	Observe now that we cannot use Theorem \ref{relaxation gsbvp} to represent $\mathcal{G}_0^w$ as an integral functional, as no $p$-growth from the above is satisfied by the bulk energy. Following the lines of \cite{hafsa2006nonlinear} and \cite{benbe1996}, our aim is to prove that $\overline{\mathcal{G}_0^w}=\overline{\mathcal{G}_0^{\mathcal{R}}}$, where
	\begin{equation}\label{500}
		\mathcal{G}_0^{\mathcal{R}}(u) := 
		\begin{cases*}
			\displaystyle \int_{\Sigma} \mathcal{R}W_0(\nabla u) \, dx +\mathcal{H}^1(J_u) & if $u \in \gsbvp(\Sigma,\R^3)$, \\
			+\infty & otherwise.
		\end{cases*}
	\end{equation} 
	Above, $\mathcal{R}W_0$ being the rank one convexified of $W_0$. Indeed, in Lemma \ref{lemma rank one} we show that $\mathcal{R}W_0$ is finite valued, continuous and satisfies an uniform $p$-growth condition. This allows us to apply the integral representation of Theorem \ref{relaxation gsbvp} to $\mathcal{G}^\mathcal{R}_0$. The proof of the equality $\overline{\mathcal{G}^w_0}=\overline{\mathcal{G}^\mathcal{R}_0}$ relies on Proposition \ref{ben belg lemma}, where we make use of the density result proved in Theorem \ref{main result 1}. \\
	
	We start with some properties of the rank-one and quasi-convex envelope of $W_0$ (see Definitions \ref{rkone} and \ref{q convex env}).
	
	\begin{lem}\label{lemma rank one}
		Let $W_0 \in C(\mathbb{M}^{3 \times 2},[0,+\infty])$ a $p$-coercive function satisfying (ii) and (iii) of Lemma \ref{reduced density lemma}. Then
		\begin{enumerate}
			\item[\textnormal{(i)}] $\mathcal{R}W_0$ is finite valued, continuous and $p$-coercive,
			\item[\textnormal{(ii)}] there exists $C>0$ such that $\mathcal{R}W_0(A) \leq C(1+|A|^p)$ for every $A \in \mathbb{M}^{3 \times 2}$,
			\item[\textnormal{(iii)}] $\mathcal{Q}W_0$ is finite valued, continuous, $p$-coercive and rank-one convex.
		\end{enumerate}
	\end{lem}
	
	\begin{proof}
		Part (i) and (ii) of the statement are proven in \cite[Proposition 7 and Lemma 8]{benbe1996}. See also \cite[Lemma 6.5]{Trabelsimodeling} and \cite[Section 5.1]{belgacem2000}.
		
		Concerning part (iii), $p$-coercivity is obvious. The fact $\mathcal{Q}W_0$ is finite valued, continuous and has $p$-growth from above is proved in \cite[Theorem 1 and Proposition 1]{HafsaMandallenaRelaxation}. By \cite[Theorem 1.1]{DacorognaDirect}, this also gives that $\mathcal{Q}W_0$ is rank-one convex .
	\end{proof}

	We will also need the following result which follows from Kohn and Strang \cite[Section 5C]{KohnStrang}. 
	
	\begin{lem}\label{KS lemma}
		Define the sequence $\left\{ \mathcal{R}_i W_0 \right\}_{i \geq 0}$ by $\mathcal{R}_0 W_0=W_0$ and for every $i \geq 0 $ and every $A \in \mathbb{M}^{3 \times 2}$
		$$
		\mathcal{R}_{i+1}W_0(A):=\inf_{\underset{\lambda \in [0,1]}{a \in \R^2, \ b \in \R^3}} \left\{ (1-\lambda)\mathcal{R}_i W_0(A-\lambda b \otimes a)+\lambda \mathcal{R}_i W_0(A+(1-\lambda) b \otimes a) \right\}.
		$$
		We have that $\mathcal{R}_iW_0$ is upper semi-continuous for every $i \geq 0$ and finite valued for every $ i \geq 2$. Moreover, $\mathcal{R}_{i+1} W_0 \leq \mathcal{R}_{i}W_0$ for every $i \geq 0$ and $\mathcal{R}W_0= \inf_{i \geq 0} \mathcal{R}_i W_0$.
	\end{lem}
	\begin{proof}
		 Notice that since $W_0$ satisfies (i) and (iii) of Lemma \ref{reduced density lemma}, we have that $\mathcal{R}_1 W_0(A)$ is finite for every $A \in \mathbb{M}^{3 \times 2}$ such that $A \neq 0$, thus $\mathcal{R}_2 W_0$ is finite valued and $\mathcal{R}_2 W_0 \leq \mathcal{R} W_0$. This allows us to apply the iterative scheme of Khon and Strang to $W_0$ even if it is not finite valued. For the upper semicontinuity, it is enough to notice that the infimum of upper semi-continuous functions is upper semi-continuous as well and reason by induction.
	\end{proof}

	As in \cite{hafsa2006nonlinear} the following technical lemma will be crucial to the asymptotic analysis. We give the proof of this result in the Appendix.
	
	\begin{lem}[Belgacem]\label{BB lemma}
		Let $V \subset \Sigma$ an open set with $|\partial V|=0$ and $A \in \mathbb{M}^{3 \times 2}$ with $\textnormal{rank}(A)=2$.  There exists $\left\{ \phi_{n,l,q} \right\}_{n,l,q \geq 1} \subset \textnormal{Aff}_c(V,\R^3)$ such that:
		\begin{enumerate}
			\item[\textnormal{(i)}] for every $l,q \geq 1$ we have $\lim_{n \to +\infty} \phi_{n,l,q} =0$ in $L^p(V,\R^3)$,
			\item[\textnormal{(ii)}] for every $n,l,q \geq 1$ the function $x \mapsto \phi_{n,l,q}(x)+A x+c \in \Aff^*(V,\R^3)$ for any $c \in \R^3$,
			\item[\textnormal{(iii)}] $\displaystyle \lim_{q \to +\infty} \lim_{l \to+ \infty} \lim_{n \to +\infty} \int_V \mathcal{R}_i W_0(\nabla \phi_{n,l,q} + A) \, dx \leq |V|\mathcal{R}_{i+1} W_0(A)$.
		\end{enumerate}
	\end{lem}
	
	We now show the $\Gamma$-limsup inequality in terms of the auxiliary functional $\mathcal{G}_0^{\mathcal{R}}$.
	
	\begin{prop}\label{ben belg lemma}
		Let $\mathcal{G}_0^{\mathcal{R}}$ be given by \eqref{500}. For every $u \in \gsbvp(\Sigma,\R^3)$ we have
		\begin{equation}
			\Gamma-\limsup_{\rho \to 0} \rho^{-1}\mathcal{G}_\rho(u) \leq \overline{\mathcal{G}_0^\mathcal{R}}(u).
		\end{equation}
	\end{prop}
	
	\begin{proof}
		We divide the proof in two steps.
		
		\noindent \textbf{Step 1:} We first prove that for every $v \in \Aff^*(\Sigma \setminus \overline{J_v},\R^3) \cap \widehat{\mathcal{W}}(\Sigma,\R^3)$, we have
		\begin{equation}\label{ben belg ineq}
			\overline{\mathcal{G}_0^{w}}(v) \leq  \int_\Sigma \mathcal{R}W_0(\nabla v) \, dx + \mathcal{H}^1(J_v).
		\end{equation}
		According to Lemma \ref{KS lemma}, it is enough to show that for every $i \geq 0$ and for every $v \in \Aff^*(\Sigma \setminus \overline{J_v},\R^3) \cap \widehat{\mathcal{W}}(\Sigma,\R^3)$ we have
		\begin{equation}\label{501}
			\overline{\mathcal{G}_0^{w}}(v) \leq \int_\Sigma \mathcal{R}_i W_0(\nabla v) \, dx + \mathcal{H}^1(J_v).
		\end{equation}
		We show \eqref{501} by induction on $i$. Since $\mathcal{R}_0W_0=W_0$ the inequality is satisfied by definition of $\mathcal{G}_0^w$ if $i=0$. Assume now that \eqref{501} holds for a certain $i \in \N$. By definition of $v$, there exists a finite family $\left\{ V_j \right\}_{j=1}^M$ of open disjoint subsets of $\Sigma$ such that $|\partial V_j|=0$ for every $j$, $|\Sigma \setminus \cup_{j=1}^M V_j|=0$, $J_v \subseteq \cup_{j=1}^M \partial V_j$ and for every $j=1,\dots ,M$ we have $ v(x)=A_jx+c$ in $V_j$, where $A_j \in \mathbb{M}^{3 \times 2}$, rank$(A_j)=2$ and $c \in \R^3$. For every $j=1,\dots ,M$ consider $\left\{ \phi^j_{k,l,q} \right\}_{k,l,q \geq 1} \subset \textnormal{Aff}_c(V_j,\R^3)$ as in Lemma \ref{BB lemma} applied with with $V=V_j$ and $A=A_j$. Now set 
		$$
		\Psi_{k,l,q}(x):=v(x)+\phi_{k,l,q}^j(x) \ \ \ \mbox{if $x \in V_j$}.
		$$
		We have $J_{\Psi_{k,l,q}}=J_v$ and $\Psi_{k,l,q} \in \Aff^*(\Sigma \setminus \overline{J_v},\R^3) \cap \widehat{\mathcal{W}}(\Sigma,\R^3)$ for every $k,l,q \geq 1$ by (ii) of Lemma \ref{BB lemma}. Hence, the inductive hypothesis implies that
		$$
		\overline{\mathcal{G}_0^{w}}(\Psi_{k,l,q}) \leq \int_\Sigma \mathcal{R}_i W_0(\nabla \Psi_{k,l,q}) \, dx+\mathcal{H}^1(J_{\Psi_{k,l,q}}) \ \ \mbox{for every $k,l,q \geq 1$}.
		$$
		By Lemma \ref{BB lemma} (i) for every $l,q\geq 1$ we see that $\Psi_{k,l,q} \to v$ in $L^p(\Sigma,\R^3)$ for $k \to \infty$. By lower semicontinuity of $\overline{\mathcal{G}_0^w}$, it follows that
		$$
		\overline{\mathcal{G}_0^{w}}(v) \leq \liminf_{k \to +\infty} \overline{\mathcal{G}_0^{w}}(\Psi_{k,l,q}) \leq \lim_{k \to +\infty} \int_\Sigma \mathcal{R}_i W_0(\nabla \Psi_{k,l,q}) \, dx+\mathcal{H}^1(J_v) \ \ \mbox{for every $k,l,q \geq 1$}.
		$$
		Moreover, from Lemma \ref{BB lemma} (iii) we have
		\begin{align*}
			\lim_{q \to +\infty} \lim_{l \to +\infty} \lim_{k \to +\infty} \int_\Sigma \mathcal{R}_i W_0(\nabla \Psi_{k,l,q}) \, dx &  = \lim_{q \to+ \infty} \lim_{l \to+ \infty} \lim_{k \to +\infty} \sum_{j=1}^M \int_{V_j} \mathcal{R}_i W_0(\nabla \phi_{k,l,q}+A_j) \, dx \\ 
			& \leq \sum_{j=1}^M |V_j|\mathcal{R}_{i+1} W_0(A_j) = \int_\Sigma \mathcal{R}_{i+1} W_0(\nabla v) \, dx.
		\end{align*}
		Therefore,
		$$
		\overline{\mathcal{G}_0^{w}}(v) \leq \int_\Sigma \mathcal{R}_{i+1}W_0(\nabla v) \, dx + \mathcal{H}^1(J_v),
		$$
		and \eqref{501} holds for all $i \in \N$. Hence, \eqref{ben belg ineq} follows.
		
		\noindent \textbf{Step 2:} We now come to the general case of $u \in \gsbvp(\Sigma,\R^3)$. Recall that, by Theorem \ref{gamma limsup prop}, we have that for every $u \in \gsbvp(\Sigma,\R^3)$
		\begin{equation}\label{503}
			\Gamma-\limsup_{\rho \to 0} \rho^{-1} \mathcal{G}_\rho(u) \leq \overline{\mathcal{G}_0^w}(u). 
		\end{equation}
		We now claim that, given $u \in \gsbvp(\Sigma,\R^3)$, it holds
		\begin{equation}\label{502}
			\overline{\mathcal{G}_0^w}(u) \leq \int_\Sigma \mathcal{R}W_0(\nabla u) \, dx+ \mathcal{H}^1(J_u).
		\end{equation}
		Indeed, by Theorem \ref{main result 1} there exists $ u_k \in \Aff^*(\Sigma \setminus \overline{J_{u_k}},\R^3) \cap \widehat{\mathcal{W}}(\Sigma,\R^3)$ such that $u_k \to u$ in measure on $\Sigma$, $\nabla u_k \to \nabla u$ in $L^p(\Sigma,\mathbb{M}^{3 \times 2})$ and
		$$ \limsup_{k \to +\infty} \mathcal{H}^{1}(J_{u_k}) \leq \mathcal{H}^{1}(J_{u}). $$ 
		By (i) and (ii) of Lemma \ref{lemma rank one} and by the $L^p$-convergence of $\nabla u_k$ to $\nabla u$ we have that
		$$
		\lim_{k \to +\infty} \int_\Sigma \mathcal{R}W_0(\nabla u_k) \, dx = \int_\Sigma \mathcal{R}W_0(\nabla u) .
		$$
		Therefore, using inequality \eqref{ben belg ineq} for the functions $u_k$, we deduce
		\begin{align*}
			\overline{\mathcal{G}_0^w}(u) & \leq \liminf_{k \to +\infty} \overline{\mathcal{G}_0^w}(u_k) \leq \liminf_{k \to +\infty} \int_\Sigma \mathcal{R}W_0(\nabla u_k) \, dx+ \mathcal{H}^1(J_{u_k}) \\ & \leq \int_\Sigma \mathcal{R}W_0(\nabla u) \, dx+ \mathcal{H}^1(J_u), 
		\end{align*}
		proving \eqref{502}. It follows that $\overline{\mathcal{G}_0^w}(u) \leq \overline{\mathcal{G}_0^{\mathcal{R}}}(u)$. Hence, we conclude using \eqref{503}. 
	\end{proof}
	
	We are finally in a position to give the proof of Theorem \ref{gamma limit teo}.
	
	\begin{proof}[Proof of Theorem \ref{gamma limit teo}]
		By (i) and (ii) in Lemma \ref{lemma rank one}, we are allowed to apply Theorem \ref{relaxation gsbvp} to the functional $\overline{\mathcal{G}_0^R}$, yielding that, for every $u \in \gsbvp(\Sigma,\R^3)$,
		$$
		\Gamma-\limsup_{\rho \to 0} \rho^{-1} \mathcal{G}_\rho(u) \leq \int_{\Sigma} \mathcal{Q}(\mathcal{R}W_0)(\nabla u) \, dx + \mathcal{H}^1(J_u).
		$$
		Combining with Proposition \ref{gamma liminf}, we deduce
		$$
		\int_{\Sigma} \mathcal{Q}(\mathcal{R}W_0)(\nabla u) \, dx + \mathcal{H}^1(J_u) \leq \mathcal{G}_0(u) \leq \Gamma-\liminf_{\rho \to 0} \rho^{-1} 	\mathcal{G}_\rho(u).
		$$
		Therefore, for every $u \in \gsbvp(\Sigma,\R^3)$, it holds
		$$
		\Gamma-\lim_{\rho \to 0} \rho^{-1} \mathcal{G}_\rho(u)=\int_{\Sigma} \mathcal{Q}(\mathcal{R}W_0)(\nabla u) \, dx + \mathcal{H}^1(J_u).
		$$
		
		It remains to show that $\mathcal{Q}W_0=\mathcal{Q}(\mathcal{R}W_0)$. The "$\geq$" inequality is obvious. By (iii) in Lemma \ref{lemma rank one} and the definition of rank-one convex envelope, we have that $\mathcal{Q}W_0 \leq \mathcal{R}W_0$. With this and Remark \ref{rem on qconv env} applied to $\mathcal{R}W_0$, we deduce $\mathcal{Q}W_0 \leq \mathcal{Q}(\mathcal{R}W_0)$. This concludes the proof of the theorem.
	\end{proof}

	\section{Appendix}
	
	This section is devoted to give a sketch of the proofs of Lemma \ref{segments lemma} and Lemma \ref{BB lemma}.
	
	\begin{proof}[Proof of Lemma \ref{segments lemma}.]
		Since we reason component-wise, we only consider the case $m=1$. Suppose that $\overline{J_u}$ is composed by the union of $K$ closed segments $\Pi$ and $\left\{ \Pi_i \right\}_{i=1}^{K-1}$ with $\Pi \subset \left\{ x_2=0 \right\}$ and that $\Pi \cap (\cup_{i=1}^{K-1} \Pi_i)=\left\{ P_j \right\}_{j=1}^M$. 
		We denote by $\Gamma_j$ the $M+1$ connected components of $\Pi \setminus \left\{ P_j \right\}_{j=1}^M$. Fix $0<\lambda <<\eta$. For each $j=1,\dots,M+1$, we denote by $Z_j$ each set composed by the union of a segment $\Gamma_j^\lambda:=\left\{x \in \Gamma_j | \ \dist(x,\cup_{j=1}^M P_j) \geq \lambda\right\}$ and two segments $S_{j}^1$ and $S_{j}^2$ of length $\eta$, with an endpoint coinciding with one endpoint of $\Gamma_j^\lambda$ and the other contained in $\left\{ x_2>0 \right\}$ (see Figure \ref{f:1}). The same construction as in \cite[Lemma 5.2]{Philippis2017OnTA} shows that, provided $\lambda$ and $\eta$ small enough we can construct a function $v \in \mathcal{W}(\Omega) \cap C^1(\Omega \setminus \overline{J_{v}}) $ such that 
		\begin{align*}
			& \overline{J_{v}}= \left(\bigcup_{i=1}^{K-1} \Pi_i\right) \cup \left(\bigcup_{j=1}^{M+1} Z_j\right), \\
			& \Vert v - u \Vert_{L^1} < \frac{\varepsilon}{2K}, \ \ \ \Vert \nabla v - \nabla u \Vert_{L^p}<\frac{\varepsilon}{2K}, \ \ \ \mathcal{H}^1(J_u \Delta J_{v}) <\frac{\varepsilon}{2K}.
		\end{align*}
		Notice that, for some constant $C>0$ independent of $\eta $ and $\lambda$, one can also ensure
		\begin{equation} \label{800} 
			\Vert \hat{u} \Vert_{W^{1,\infty}} \leq  \frac{C \Vert u \Vert_{W^{1,\infty}}}{\eta}.
		\end{equation}
		
		Now we want to exploit the same construction in order to transform each portion of the jump $Z_1$ into three disjoint components, the first being given by $\Gamma_1^\lambda$ and the remaining two being arbitrarily small "L-shaped" sets entirely contained in $\left\{ x_2 >0\right\}$. This can be done again by fixing $0<\mu<<\delta<<\eta$ (see again Figure \ref{f:1}) and repeating the construction of \cite[Lemma 5.2, proof of (A.8)]{Philippis2017OnTA}. This results in a function $v_1$ with 
		\begin{align}
			\nonumber
			& \overline{J_{v_1}}= \left(\bigcup_{i=1}^{K-1} \Pi_i\right) \cup \left(\bigcup_{j=2}^{M+1} Z_j\right) \cup \Gamma_1^\lambda \cup L_1 \cup L_2, \\
			\label{801}
			& \Vert v_1 - v \Vert_{L^1} < \frac{\varepsilon}{2K(M+1)}, \ \ \ \Vert \nabla v_1 - \nabla v \Vert_{L^p}<\frac{\varepsilon}{2K(M+1)}, \ \ \ \mathcal{H}^1(J_{v_1} \Delta J_v) <\frac{\varepsilon}{2K(M+1)} .
		\end{align}
		Without entering in details (for which we refer again to \cite[Lemma 5.2]{Philippis2017OnTA}), we mention here that a proper choice of $\mu $ and $\delta$ in dependence of \eqref{800} is needed to bound the $L^p$ distance of the gradients.
		Finally, by iterating this procedure for $j=1, \dots, M+1$ and then $i=1,\dots,K-2$, and summing on the inequalities in \eqref{801}, one can obtain the required approximation $u_\varepsilon \in \widehat{\mathcal{W}}(\Omega)$.
		
	\end{proof}

	\begin{figure}[h!]
		\begin{tikzpicture}
			\draw[black, very thick] (0,0) rectangle (5,5);
			\draw[thick] (0,0) rectangle (3.8, 3.4);
			\node at (0.8, 0.5) {$ u = 0$};
			\node at (0.8, 4.5) {$ u = 1$};
			\draw[ very thick] (7,0) rectangle (12,5);
			\draw[thick] (10.8, 0) -- (10.8, 3.4);
			\draw[thick] (7, 3.4) -- (10, 3.4); 
			\draw[thick] (10, 3.4) -- (10, 1.6);
			\node at (7.8, 0.5) {$ u = 0$};
			\node at (7.8, 4.5) {$ u = 1$};
			\fill[fill= gray, fill opacity=0.3] (10.8, 3.4) rectangle (10, 1.6);
			\draw [decorate,decoration={brace, amplitude=4pt}, xshift=0pt,yshift=2pt]
			(10, 3.4) -- (10.8, 3.4)  node [black,midway,xshift=0cm, yshift=4mm] {  $\lambda \eta$};
			\draw [decorate,decoration={brace, amplitude=4pt}, xshift=2pt,yshift=0pt]
			(10.8, 3.4) -- (10.8, 1.6)  node [black,midway,xshift=4mm, yshift=0mm] {  $ \eta$};
			\draw[ very thick] (0,-7) rectangle (5,-2);
			\draw[thick] (3.8, -7) -- (3.8, -3.6);
			\draw[thick] (0, -3.6) -- (3, -3.6); 
			\draw[thick] (3, -4.2) -- (3, -5.4);
			\draw[thick] (2, -4.2) -- (3, -4.2);
			\node at (0.8, -6.5) {$ u = 0$};
			\node at (0.8, -2.5) {$ u = 1$};
			\fill[fill= gray, fill opacity=0.3] (3.8, -3.6) rectangle (3, -5.4);
			\fill[fill= gray, fill opacity=0.3] (2, -4.2) rectangle (3, -3.6);
			\draw [decorate,decoration={brace, amplitude=4pt}, xshift=-2pt,yshift=0pt]
			(2,-4.2) -- (2, -3.6)  node [black,midway,xshift=-4mm, yshift=0mm] {  $ \mu \delta$};
			\draw [decorate,decoration={brace, amplitude=4pt}, xshift=0pt,yshift=2pt]
			(2, -3.6) -- (3, -3.6)   node [black,midway,xshift=0mm, yshift=4mm] {  $ \delta$};
		\end{tikzpicture}
		\caption{} \label{f:1}
	\end{figure}

	Since Lemma \ref{BB lemma} is essentially proved in \cite{benbe1996} and \cite[Lemma B.5]{hafsa2006nonlinear}, we only give a sketch of the proof highlighting the main differences. 
	
	Before proving Lemma \ref{BB lemma}, we need the following result.
	
	\begin{lem}\label{3000}
		Let $Q$ be the unitary cube in $\R^2$, $Q:=(0,1)^2$, let $A \in \mathbb{M}^{3 \times 2}$ a full rank matrix and let $i \geq 0$ be fixed. There exists a sequence $\theta_{n,\ell} \in \Aff_c(Q,\R^3)$ such that 
		\begin{enumerate}
			\item[\textnormal{(i)}] for every $\ell \geq 1$, $\theta_{n,\ell} \to 0$ in $L^p(Q,\R^3)$ as $n \to +\infty$;  
			\item[\textnormal{(ii)}] $x \mapsto Ax+\theta_{n,\ell}(x) \in \Aff^*(Q,\R^3)$;
			\item[\textnormal{(iii)}] $\lim_{\ell \to +\infty} \lim_{n \to +\infty} \int_Q \mathcal{R}_i W_0(A+\nabla \theta_{n,\ell}) \, dx \leq \mathcal{R}_{i+1} W_0(A)$.
		\end{enumerate}
	\end{lem}
	
	\begin{proof}
		By definition, for every $i \geq 0$ the function $\mathcal{R}_i W_0$ is $p$-coercive, thus, there exist $a \in \R^2$, $b \in \R^3$ and $\lambda \in [0,1]$ such that
		\begin{equation}\label{--}
		\mathcal{R}_{i+1}W_0(A)=(1-\lambda)\mathcal{R}_i W_0(A-\lambda b \otimes a)+\lambda \mathcal{R}_i W_0(A+(1-\lambda) b \otimes a).
		\end{equation}
		Without loss of generality we can suppose $a=(0,1)$. For every $n \geq 3$ and $k \in \{ 0,\dots,n-1 \}$, consider the following subsets of $Q$:
		\begin{align*}
			& \textstyle A_{k,n}^-:=\left\{ (x,y) \in Q: \ \frac{k}{n} \leq x \leq \frac{k}{n}+\frac{1-\lambda}{n} \mbox{ and } \frac{1}{n} \leq y \leq 1-\frac{1}{n} \right\}, \\
			& \textstyle A_{k,n}^+:=\left\{ (x,y) \in Q: \ \frac{k}{n}+\frac{1-\lambda}{n} \leq x \leq \frac{k+1}{n} \mbox{ and } \frac{1}{n} \leq y \leq 1-\frac{1}{n} \right\}, \\
			& \textstyle B_{k,n}:=\left\{ (x,y) \in Q: \ \frac{k}{n} \leq x \leq \frac{k+1}{n} \mbox{ and } 0 \leq y \leq -x+\frac{k+1}{n} \right\}, \\
			& \textstyle B_{k,n}^-:=\left\{ (x,y) \in Q: \ -y+\frac{k+1}{n} \leq x \leq -\lambda y+\frac{k+1}{n} \mbox{ and } 0 \leq y \leq \frac{1}{n} \right\}, \\
			& \textstyle B_{k,n}^+:=\left\{ (x,y) \in Q: \ -\lambda y+\frac{k+1}{n} \leq x \leq \frac{k+1}{n} \mbox{ and } 0 \leq y \leq \frac{1}{n} \right\}, 
		\end{align*}
		\begin{align*}
			& \textstyle C_{k,n}:=\left\{ (x,y) \in Q: \ \frac{k}{n} \leq x \leq \frac{k+1}{n} \mbox{ and } x+1-\frac{k+1}{n} \leq y \leq 1 \right\}, \\
			& \textstyle C_{k,n}^-:=\left\{ (x,y) \in Q: \ y-1+\frac{k+1}{n} \leq x \leq \lambda(y-1)+ \frac{k+1}{n} \mbox{ and } 0 \leq y \leq \frac{1}{n} \right\}, \\
			& \textstyle C_{k,n}^+:=\left\{ (x,y) \in Q: \ \lambda(y-1)+\frac{k+1}{n} \leq x \leq \frac{k+1}{n} \mbox{ and } 0 \leq y \leq \frac{1}{n} \right\}.
		\end{align*}
		We define the sequence $\sigma_n \in \Aff_c(Q,\R^3)$ as
		\begin{equation*}
			\sigma_n(x,y):=
			\begin{cases*}
				-\lambda(x-\frac{k}{n}) & if $(x,y) \in A_{k,n}^-$, \\
				(1-\lambda)(x-\frac{k+1}{n}) & if $(x,y) \in A_{k,n}^+ \cup B^+_{k,n} \cup C^+_{k,n}$, \\
				-\lambda(x+y-\frac{k+1}{n}) & if $(x,y) \in B_{k,n}^-$, \\
				-\lambda(x-y+1-\frac{k+1}{n}) & if $(x,y) \in C_{k,n}^-$, \\
				0 & if $(x,y) \in B_{k,n} \cup C_{k,n}$.
			\end{cases*}
		\end{equation*}
		Let $\nu$ be a normal vector to the space $\textnormal{Im} \ A:=\{ Ax : \ x \in \R^2 \}$ and $\ell \geq 1$, set
		\begin{equation*}
			b_\ell:=
			\begin{cases*}
				b & if $b \notin \textnormal{Im} \ A$, \\
				b+\frac{1}{\ell} \nu & if $b \in \textnormal{Im} \ A$.
			\end{cases*}
		\end{equation*}  
		We are finally ready to define the sequence $\theta_{n,\ell}$ as $\theta_{n,\ell}:=\sigma_n b_\ell$.
		
		(i) Observe that the following estimate holds for every $n \geq 3$
		$$
		\int_Q |\sigma_n(x)|^p \,dx \leq \frac{\lambda^p(1-\lambda)^p}{n^p};
		$$
		therefore, $\theta_{n,\ell} \to 0$ in $L^p$ as $n \to +\infty$.
		
		(ii) Since we have 
		\begin{equation*}
			\nabla \theta_{n,\ell}(x,y)=
			\begin{cases*}
				-\lambda b_\ell \otimes a & if $(x,y) \in A_{k,n}^-$, \\
				(1-\lambda)b_\ell \otimes a & if $(x,y) \in A_{k,n}^+ \cup B^+_{k,n} \cup C^+_{k,n}$, \\
				-\lambda b_\ell \otimes (a+a^\perp) & if $(x,y) \in B_{k,n}^-$, \\
				-\lambda b_\ell \otimes (a-a^\perp) & if $(x,y) \in C_{k,n}^-$, \\
				0 & if $(x,y) \in B_{k,n} \cup C_{k,n}$;
			\end{cases*}
		\end{equation*}
		the admissible gradients of the affine function $x \mapsto Ax+\theta_{n,\ell}(x)$ belong to a finite set of the form
			$$
			\{ A+b_\ell \otimes d_i : \ i=1,\dots,5 \}.
			$$
		The convex hull of the above set of gradients is clearly made of injective matrices because $\textnormal{rank}(A)=2$ and $b_\ell \notin \textnormal{Im} \ A$ for every $\ell \geq 1$. Thus $x \mapsto Ax+\theta_{n,\ell}(x) \in \Aff^*(Q, \R^3)$.
		
		(iii) After a direct computation for every $n \geq 3$ and $\ell \geq 1$, we get
		\begin{align*}
			\int_Q \mathcal{R}_i W_0(A+ \nabla \theta_{n,\ell}(x)) \, dx = & \left( 1-\frac{2}{n} \right) \left[ \vphantom{\frac{1}{1}} (1-\lambda)\mathcal{R}_i W_0(A-\lambda b_\ell \otimes a)  +\lambda\mathcal{R}_i W_0(A+ \right. \\ & \left. (1-\lambda) b_\ell \otimes a) \vphantom{\frac{1}{1}} \right]+\frac{1}{n} \left[ \lambda\mathcal{R}_i W_0(A+(1-\lambda) b_\ell \otimes a)+ \vphantom{\frac{1}{1}} \right. \\ & \left. \frac{1-\lambda}{2}(\mathcal{R}_i W_0(A  -\lambda b_\ell \otimes (a+a^\perp))+ \right. \\ & \left. \mathcal{R}_i W_0(A-\lambda b_\ell \otimes (a-a^\perp))) \vphantom{\frac{1}{1}} +\mathcal{R}_i W_0(A) \right].
		\end{align*}
		It follows that for every $\ell \geq 1$,
		\begin{align*}
			\lim_{n \to +\infty} \int_Q \mathcal{R}_i W_0(A+\nabla \theta_{n,\ell}) \, dx=(1-\lambda)\mathcal{R}_i W_0(A-\lambda b_\ell \otimes a)+\lambda \mathcal{R}_i W_0(A+(1-\lambda) b_\ell \otimes a).
		\end{align*}
		Recalling that $\mathcal{R}_i W_0$ is upper semi-continuous and \eqref{--}, we conclude.
	\end{proof}

	\begin{proof}[Proof of Lemma \ref{BB lemma}.]
		Since $|\partial V| =0$, there exists a sequence of plurirectangles $V_q$ compactly contained in $V$ and such that $|V \setminus V_q|\to 0$ as $q \to +\infty$. Fix $q \geq 1$. We have that $V_q=\cup_{m=1}^{M_q} (r_m+\rho_m Q)$, where $M_q \in \N$, $\rho_m >0$ and $r_m \in \R^2$ for every $m=1,\dots,M_q$. We define
		\begin{equation*}
			\phi_{k,l,q}:=
			\begin{cases*}
				\rho_m \theta_{n,l}\left(\frac{x-r_m}{\rho_m}\right) & if $x \in r_m+\rho_m Q \subset V_q$, \\
				0 & if $x \in V \setminus V_q$;
			\end{cases*}
		\end{equation*}
		where $\theta_{n,l}$ is the sequence given by Lemma \ref{3000}. We notice that since $$\nabla \phi_{n,l,q}(x)=\nabla \theta_{n,l}\left(\frac{x-r_m}{\rho_m}\right)$$ for every $x \in r_m+\rho_m Q$, (ii) follows by (ii) of Lemma \ref{3000}. Moreover, also (i) follows directly from (i) of Lemma \ref{3000}.
		
		We are left to show (iii). Recalling that $\sum_{m=1}^{M_q} \rho_m^2=|V_q|$, we compute by rescaling
		\begin{align*}
			\int_V \mathcal{R}_i W_0(A+\nabla \phi_{n,l,q}(x)) \, dx &=\int_{V_q} \mathcal{R}_i W_0(A+\nabla \phi_{n,l,q}(x)) \,  dx+|V \setminus V_q|\mathcal{R}_i W_0(A) \\ & = |V_q|\int_Q \mathcal{R}_i W_0(A+\nabla \theta_{n,l}(x)) \,  dx+|V \setminus V_q|\mathcal{R}_i W_0(A).
		\end{align*} 
		Using Lemma \ref{3000} (iii), we deduce that
		$$
		\lim_{l \to +\infty} \lim_{n \to +\infty} \int_V \mathcal{R}_i W_0(A+\nabla \phi_{n,l,q}(x)) \, dx \leq |V_q|\mathcal{R}_{i+1}W_0(A)+|V \setminus V_q|\mathcal{R}_i W_0(A),
		$$
		and (iii) follows by letting $q \to +\infty$.
%
	\end{proof}
	
\section*{Acknowledgments} \noindent The authors wish to thank N. Fusco for a very useful discussion.

\noindent The work of S. Almi was partially funded by the FWF Austrian Science Fund through the Project ESP-61

\noindent The work of F. Solombrino is part of the MIUR - PRIN 2017, project Variational Methods for Stationary and Evolution Problems with Singularities and Interfaces, No. 2017BTM7SN.  He also acknowledges support by project Starplus 2020 Unina Linea 1 "New challenges in the variational modeling of continuum mechanics" from the University of Naples Federico II and Compagnia di San Paolo, and by Gruppo Nazionale per l’Analisi Matematica, la Probabilit\`a e le loro Applicazioni (GNAMPA-INdAM).

	\bibliographystyle{siam}
	\bibliography{bibliography_NEW}
	
\end{document}